\documentclass[11pt]{article}

\usepackage{comment}

\usepackage[round]{natbib} 
\usepackage{fullpage}
\usepackage{subcaption}

\usepackage{graphicx}
\usepackage{enumitem}

\usepackage{xcolor}
\usepackage[colorlinks=true,linkcolor=blue,citecolor=blue]{hyperref}

\usepackage{amsthm}
\usepackage{amsmath}
\usepackage{amsfonts}
\usepackage{amssymb}
\usepackage{euscript}
\usepackage{bm}

\newcommand{\D}{\,\mathrm{d}} 
\newcommand{\one}{\mathbf{1}}
\renewcommand{\subset}{\subseteq}

\newcommand{\bfx}{\textbf{x}}

\newcommand{\bB}{\mathbf{B}}
\newcommand{\bU}{\mathbf{U}}

\newcommand{\bzero}{\mathbf{0}}

\newcommand{\be}{\mathbf{e}}
\newcommand{\bh}{\mathbf{h}}
\newcommand{\bp}{\mathbf{p}}

\newcommand{\bu}{\mathbf{u}}
\newcommand{\bx}{\mathbf{x}}
\newcommand{\by}{\mathbf{y}}

\newcommand{\bN}{\mathbf{N}}
\newcommand{\bX}{\mathbf{X}}

\newcommand{\cD}{\mathcal{D}}
\newcommand{\cF}{\mathcal{F}}

\newcommand{\cX}{\mathcal{X}}
\newcommand{\cY}{\mathcal{Y}}
\newcommand{\cZ}{\mathcal{Z}}

\newcommand{\sP}{\EuScript{P}}
\newcommand{\sS}{\EuScript{S}}

\newcommand{\sH}{\EuScript{H}}

\newcommand{\bbP}{\mathbb{P}}
\newcommand{\bbQ}{\mathbb{Q}}
\newcommand{\bbE}{\mathbb{E}}

\newcommand{\bbR}{\mathbb{R}}
\newcommand{\ebbR}{\overline{\mathbb{R}}}
\newcommand{\crps}{\operatorname{CRPS}}

\DeclareMathOperator*{\argmin}{arg\,min}

\theoremstyle{definition}
\newtheorem{definition}{Definition}
\newtheorem{remark}[definition]{Remark}
\newtheorem{example}[definition]{Example}

\theoremstyle{plain}

\newtheorem{theorem}[definition]{Theorem}

\usepackage{multirow}
\usepackage{multicol}

\definecolor{kartik}{rgb}{0.8,0, 0.2}

\title{\bfseries Proper scoring rules for estimation\\ and forecast evaluation}
\author{Kartik Waghmare and Johanna Ziegel\thanks{ETH Zurich, Seminar for Statistics, Switzerland, \url{kartik.waghmare@math.ethz.ch,ziegel@stat.math.ethz.ch}}}
\date{}

\begin{document}

\maketitle

\begin{abstract}
    Proper scoring rules have been a subject of growing interest in recent years, not only as tools for evaluation of probabilistic forecasts but also as methods for estimating probability distributions.
    In this article, we review the mathematical foundations of proper scoring rules including general characterization results and important families of scoring rules. We discuss their role in statistics and machine learning for estimation and forecast evaluation. Furthermore, we comment on interesting developments of their usage in applications.
\end{abstract}

\section{Introduction}

In recent years, proper scoring rules have emerged as a powerful general approach for estimating probability distributions. In addition to significantly expanding the range of modeling techniques that can be applied in practice, this has also substantially broadened the conceptual understanding of estimation methods. Originally, proper scoring rules were conceived in meteorology as summary statistics for describing the performance of probabilistic forecasts \citep{Murphy1984}, but they also play an important role in economics as tools for belief elicitation \citep{Schotter2014}.

A probabilistic forecast is a probability distribution over the space of the possible outcomes of the future event that is stated by the forecaster. The simplest and most popular case of probabilistic forecasts arises when the outcome is binary, so the probabilistic forecast reduces to issuing a predictive probability of success. \citet{brier1950} was the first to consider the problem of devising a scoring rule which could not be ``played'' by a dishonest forecasting agent. 
He introduced the quadratic scoring rule and showed that it incentivizes a forecasting agent to state his most accurate probability estimate when faced with uncertainty.
\citet{mccarthy1956}\footnote{better known as the designer of the Lisp programming language.} was the first to clearly state the concept of propriety as it is understood today and stated as Equation \ref{eq:propriety}. 

Consider a forecasting tournament which consists of a forecasting agent issuing a probabilistic forecast $\bbP$ for a random variable $Y$ and then receiving a penalty $S(\bbP, y)$ according to the forecast $\bbP$ and the realization $y$ of $Y$. If the forecasting agent is rational 
and believes that $Y \sim \bbQ$ for some $\bbQ$, then it would issue the forecast $\bbP$ so as to mimimize the expected value $S(\bbP, \bbQ) = \bbE_{Y \sim \bbQ}[S(\bbP, Y)]$ of its penalty. In order to compel the agent to reveal its true belief $\bbQ$, we must choose $S$ such that $S(\bbP, \bbQ)$ is minimized at $\bbP = \bbQ$. It follows that we must have
\begin{equation}\label{eq:propriety}
    S(\bbQ, \bbQ) \leq S(\bbP, \bbQ)
\end{equation}
for every $\bbP, \bbQ$ (with equality if and only if $\bbP = \bbQ$). In other words, $S$ should be a \emph{(strictly) proper scoring rule}, see Definition \ref{def:1}.

\begin{example}[Brier score]\label{ex:Brier}
For a binary outcome $y \in \{0,1\}$, the \emph{Brier} or \emph{quadratic score}
\[
S_{\mathrm{Brier}}(\bbP,y) = (p - y)^2,
\]
where $p$ is the predicted probability for the event $\{y=1\}$, is a strictly proper scoring rule \citep{brier1950}. 
\end{example}

\begin{example}[Logarithmic score]\label{ex:logarithmic}
   The \emph{logarithmic score}  is given by
    \begin{equation*}
        S_{\mathrm{log}}(\bbP, y) = -\log p(y),
    \end{equation*}
    where $p$ is the density of $\bbP$. It is also known as the ignorance score in meteorology and is a strictly proper scoring rule \citep{Good1952}, see Example \ref{ex:logarithmic2} for details. 
\end{example}

\begin{example}[Continuous ranked probability score (CRPS)]\label{ex:CRPS}
For a real-valued outcome $y \in \bbR$, the \emph{continuous ranked probability score (CRPS)} is defined as
\[
\crps(\bbP,y) = \int \left(F_\bbP(x) - \one\{y \le x\}\right)^2 \D x,
\]
where $F_\bbP$ is the cumulative distribution function (CDF) of $\bbP$ \citep{MathesonWinkler1976}. It is a strictly proper scoring rule. See Example \ref{ex:CRPS2} for details. 
\end{example}

During the 1970s, proper scoring rules were developed into a principled approach for the assessment of probabilistic forecasts with applications in meteorology, economics and psychology \citep{Murphy1970} and as a foundation for the theory of subjective probability \citep{Savage1971}. \citet{mccarthy1956} produced a general characterization of proper scoring rules for finite outcome spaces.
His characterization was later extended by \citet{Hendrickson1971} to absolutely continuous measures on the real line, which was followed by a complete characterization by \citet{gneiting2007} for general measurable spaces, see Theorem \ref{thm:gneiting-char}. 

Proper scoring rules have long been recognized as elicitation and aggregation mechanisms for subjective opinion in economics and other social sciences \citep{Jensen1973}. For example, the spherical scoring originates in psychology  \citep{Roby1965}. \citet{Tetlock2005} provides a popular account of studies on the accuracy of probabilistic forecasts concerning political events by subject-matter experts. 
Another interesting application is the crowd-sourcing of forecasts from multiple human subjects via forecasting tournaments such as the Good Judgement Open and Metaculus, which evaluate forecasts using the Brier and logarithmic scores, respectively \citep{GJOpenFAQ, MetaculusScoresFAQ}. A fascinating analysis of the literature on proper scoring rules in applications from 1950 to 2015 is given by \citet{Carvalho2016}.

The extant literature contains both of the possible conventions for proper scoring rules: one which requires $S(\bbP,\bbQ)$ to be minimized at $\bbP=\bbQ$ and the other requiring it to be maximized instead. The influential paper by \citet{gneiting2007} assumes that better forecasts correspond to higher scoring rule values, but most of the more recent literature has adopted the convention in Equation \ref{eq:propriety}. With a negative orientation, scoring rules should be called loss functions, but it is not in our power to change the naming convention. 

Some parts of the literature order the arguments of proper scoring rules differently, that is, it would be $S(y,\bbP)$ instead of $S(\bbP,y)$. We mention this since it may cause confusion when dealing with expected scores $S(\bbP,\bbQ)$. 
``Scoring rule'' is sometimes abbreviated to ``score'', which may cause confusion, since ``score'' typically refers to the gradient of the logarithmic density in statistics, machine learning and econometrics. However, for some classes of proper scoring rules such as kernel scores (Section \ref{sec:kernel}) the term has now become established.

\subsection{Proper scoring rules in estimation}\label{sec:psr-estimation}

Proper scoring rules present a general and nuanced approach to the problem of estimating probability distributions, which allows for taking into account computational and statistical considerations relevant to the application.

If we think of $\bbQ$ as the data distribution and $\bbP$ as a model distribution, then Equation \ref{eq:propriety} can be interpreted as saying that the expectation $S(\cdot, \bbQ) = \bbE_{Y \sim \bbQ}[S(\cdot, Y)]$ is minimized precisely when the model and data distributions are equal. 
It is reasonable to expect that given an iid sample $\smash{\{Y_{j}\}_{j=1}^{n}}$ from $\bbQ$ with empirical distribution $\hat{\bbQ}_n$, the minimizer of the empirical mean $S(\cdot, \hat{\bbQ}_n) = (1/n) \sum_{j=1}^{n} S(\cdot, Y_{j})$ would give reasonably good estimates of $\bbQ$ under appropriate regularity conditions. Scoring rule minimization can thus be employed for estimating probability distributions.

For example, let $\{\bbP_{\theta}\}_{\theta \in \Theta}$ be a parametric family of distributions and assume that $\bbQ = \bbP_{\theta_\ast}$ for some $\theta_{\ast} \in \Theta$. Consider the estimator
\begin{equation*}
    \hat{\theta} = \argmin_{\theta \in \Theta} \left[ \frac{1}{n} \sum_{j=1}^{n} S(\bbP_{\theta}, Y_{j}) \right] = \argmin_{\theta \in \Theta} S(\bbP_\theta,\hat{\bbQ}_n),
\end{equation*}
where $S$ is a strictly proper scoring rule. This is known as \emph{minimum scoring rule inference}, of which the classical maximum likelihood inference is a special case where $S=S_{\mathrm{log}}$ is the logarithmic score of Example \ref{ex:logarithmic} \citep{dawid2016}. Strict propriety ensures that if $\hat{\theta}$ was minimizing the limit of $S(\bbP_\theta,\hat{\bbQ}_n)$ instead, which is the expectation $S(\bbP_{\theta}, \bbP_{\theta^{\ast}})=S(\bbP_{\theta}, \bbQ)$, then $\hat{\theta} = \theta_{\ast}$. Propriety can thus be thought of as the population counterpart of consistency for M-estimators of probability distributions \citep{Huber2009}.

Minimization of strictly proper scoring rules can also be used to estimate conditional distributions. Let $\{(X_{j}, Y_{j})\}_{j=1}^{n}$ be an iid sample of covariates and responses with $(X_j,Y_j) \in \cX \times \cY$, and let $\{\bbP_{\theta, x}: (\theta, x) \in \Theta \times \cX\} $ be a parametric family of distributions on $\cY$ that model the conditional distributions $\bbP_{Y|X = x}$ of the response $Y$ given covariate $X=x$. We consider
\begin{equation}\label{eq:regression}
    \hat{\theta} = \argmin_{\theta \in \Theta} \left[ \frac{1}{n} \sum_{j=1}^{n} S(\bbP_{\theta, X_{j}}, Y_{j}) \right]
\end{equation}
which yields the estimator $\smash{\bbP_{\hat{\theta}, x}}$ of $\bbP_{Y|X = x}$ for every $x \in \cX$. This approach estimates the conditional distribution simultaneously for every quantile and covariate, unlike some classical methods such as quantile regression.
In this sense, strictly proper scoring rules play the same role for estimating conditional distributions $\bbP_{Y| X= x}$, that the squared error loss plays for the conditional expectation $\bbE[Y|X = x]$. Moreover, in a certain sense, these are the only loss functions which can consistently estimate conditional distributions (see Theorem \ref{thm:regression}).

Although every proper scoring rule can identify the true model or the conditional distribution as its minimum in the population sense, they exhibit different inductive biases or regularization effects in finite samples. Furthermore, the computational tractability varies substantially depending on the chosen scoring rule, see Section \ref{sec:applications}.

\subsection{Proper scoring rules in forecast evaluation}

\citet{GneitingKatzfuss2014} have argued that forecasts should be probabilistic, since future outcomes are uncertain and this uncertainty ought to be quantified for better decision making. Probabilistic forecasts are routinely issued in more and more application domains including but not limited to weather and climate prediction, flood risk prediction, renewable energy forecasts, financial risk management, epidemiological prediction, and preventive medicine. Given competing forecasters or forecasting methods, the predictive power is compared with a proper scoring rule $S$. Suppose that there is data $(\bbP_i,\bbP'_i,y_i)_{i=1}^n$ of competing probabilistic forecasts $\bbP_i$, $\bbP_i'$ and corresponding verifying observations $y_i$ for each forecasting instance $i$ (for example, each day for daily predictions). Then, one compares the average realized scores
\begin{equation}\label{eq:empirical}
\frac{1}{n}\sum_{i=1}^n S(\bbP_i,y_i) \quad \text{and} \quad \frac{1}{n}\sum_{i=1}^n S(\bbP_i',y_i),
\end{equation}
and gives preference to the forecaster with the lower score. Justification of this approach, even in the absence of stationarity of forecast-observation pairs, can be found in \citet{ModesteDombryETAL2023}. Inference methods for differences in predictive performance are discussed by \citet{DieboldMariano1995,GiacominiWhite2006,LaiGrossETAL2011,HenziZiegel2022,ChoeRamdas2024} but this list is not comprehensive. The mentioned works focus on assessing predictive performance comparatively between two forecasters based on proper scoring rules. A comprehensive review of inference methods for assessing predictive performance in relative and absolute terms goes beyond the scope of this paper.

Clearly, the empirical scores in Equation \ref{eq:empirical} can be also be used as an evaluation metric (validation or testing loss) for comparing the performance of different estimates.

While Equation \ref{eq:propriety} ensures that the optimal forecaster (or true model) will dominate all other forecasters (models or parameters) with regards to the comparison in Equation \ref{eq:empirical} in the long term, different proper scoring rules will rank competing imperfect forecasters differently even for infinite amounts of data, see for example \citet{Patton2020}. While this may seem as a disadvantage of ranking forecasts with respect to proper scoring rules, it is not surprising and there there is no alternative. Therefore, the scoring rule must be chosen carefully for the comparison of empirical scores to be relevant to the application at hand, since it determines when one imperfect probabilistic forecast should be preferred over another. As a result, there are numerous proposals to tailor proper scoring rules such that they are discriminative with respect to certain regions of interest of the predictive distributions and outcomes, see Section \ref{sec:applications}.

\subsection{Structure of the paper}

In Section \ref{sec:properscoringrules}, we formally introduce proper scoring rules, discuss characterizations and interpret their geometric properties. Section \ref{sec:families} considers the most important families of proper scoring rules such as kernel scores, local scoring rules and $f$-scores. Scoring rule decompositions are briefly reviewed in Section \ref{sec:decomposition}. The use of proper scoring rules in estimation and forecast evaluation is discussed in Section \ref{sec:applications}, where we consider aspects of computational tractability and the choice of scoring rules for evaluation. The paper concludes with Section \ref{sec:extensions}, where we mention some limitations and extensions of proper scoring rules. All proofs of formal statements in the paper can be found in the appendix.

\section{Proper Scoring Rules}\label{sec:properscoringrules}

Let $\cY$ be a set, $\cF$ be a $\sigma$-algebra on $\cY$ and $\sP$ denote a convex set of probability measures on the measure space $(\cY, \cF)$. Let $\ebbR=\bbR \cup \{-\infty, +\infty\}$ be the extended real line. 
We refer to $\cY$ as \emph{outcome space}. When $\cY = \bbR^{d}$ for some $d \geq 1$, we assume that $\cF$ is the Borel $\sigma$-algebra.

\subsection{Definition} 

A \emph{scoring rule} is a function 
\[
S: \sP \times \cY \to \overline{\bbR}, \quad (\bbP,y) \mapsto S(\bbP,y)
\]
such that $y \mapsto S(\bbP, y)$ is $\sP$-quasi-integrable for every $\bbP \in \sP$, that is, the integral $S(\bbP,\bbQ) = \int S(\bbP,y) \D \bbQ(y)$ exists for all $\bbP,\bbQ \in \sP$ but is not necessarily finite. 
Since we are interested in minimizing $S(\bbP, \bbQ)$ over $\bbP$, multiplication by a positive constant or addition of a function of $y$ is inconsequential. Thus, two scoring rules $S_{1}$ and $S_{2}$ are said to be \emph{equivalent} if for all $y \in \cY$ and $\bbP \in \sP$, $S_{1}(\bbP, y) = \alpha S_{2}(\bbP, y) + l(y)$ for some $\alpha > 0$ and $\sP$-integrable $l: \cY \to \overline{\bbR}$, and \emph{strongly equivalent} if, in addition, $\alpha = 1$. 

\begin{definition}[Proper Scoring Rules]\label{def:1}
    A scoring rule $S: \sP \times \cY \to \overline{\bbR}$ is called \emph{proper} if 
    \begin{equation}\label{eqn:propriety}
        S(\bbQ, \bbQ) \leq S(\bbP, \bbQ)
    \end{equation}
    for every $\bbP, \bbQ \in \sP$ and \emph{strictly proper} if Equation \ref{eqn:propriety} holds with equality if and only if $\bbP = \bbQ$. 
\end{definition}

We have defined (strict) propriety of a scoring rule with regard to all measures in $\sP$. Since there could be cases when $S$ is a scoring rule naturally defined on $\sP \times \cY$ but it is only proper when restricted to some subdomain $\sP' \times \cY \subset \sP \times \cY$, it is important that the domain of a proper scoring rule is clearly defined. \citet{gneiting2007} account for this by defining propriety \emph{relative to some $\sP' \subset \sP$.} For reasons of conciseness, we prefer the formulation in Definition \ref{def:1}. 
 
\begin{example}[Logarithmic score]\label{ex:logarithmic2}
    Let $\mu$ be a $\sigma$-finite measure on $\cY$ and $\sP$ the set of measures which are absolutely continuous with respect to $\mu$. For example, if $\cY = \bbR^d$ and $\mu$ is the Lebesgue measure, we consider absolutely continuous distributions on $\bbR^d$; if $\cY$ is finite and $\mu$ is the counting measure, we have categorical random variables. The \emph{logarithmic score} $S_{\mathrm{log}}: \sP \times \cY \to \overline{\bbR}$ with 
   $S_{\mathrm{log}}(\bbP, y) = -\log p(y)$, where $p$ is the density of $\bbP$ with respect to $\mu$, is a strictly proper scoring rule \citep{Good1952}. The logarithmic score is an example of a \emph{local scoring rule}, see Section \ref{sec:local} for further details. Minimizing the logarithmic score is equivalent to the well-known maximum likelihood principle, which is the most popular method for estimating probability distributions and arguably, the foundation of classical statistical inference. 
\end{example}

\begin{example}[Hyv\"{a}rinen score]\label{ex:hyvarinnen}
    Let $\cY = \bbR^{d}$ (equipped with the Euclidean norm $\|\cdot\|$) and $\sP$ be the set of absolutely continuous distributions
    with twice continuously differentiable densities $p$ such that $\|\nabla_{\bx}\log p(\bx)\| \to 0$ as $\|\bx\| \to \infty$. The \emph{Hyv\"{a}rinen score} $S_{\mathrm{Hyv\ddot{a}rinen}}: \sP \times \cY \to \overline{\bbR}$ given by 
    \begin{equation*}
        S_{\mathrm{Hyv\ddot{a}rinen}}(\bbP, \by) = \Delta_{\by}\log p(\by) + \frac{1}{2}\|\nabla_{\by} \log p(\by)\|^{2},
    \end{equation*}
    where $\nabla_{\by}f$ and $\Delta_{\by}f$ denote the gradient $(\partial f/\partial y_{1}, \dots, \partial f/\partial y_{d})$ and Laplacian $\sum_{j=1}^{d} \partial^{2}f/\partial y_{j}^{2}$ of $f: \bbR^{d} \mapsto \bbR$ with respect to $\by = (y_{1}, \dots, y_{d})$, and 
    $p$ is the density of $\bbP$, is a strictly proper scoring rule \citep{parry2012}. Minimizing the Hyv\"{a}rinen score is equivalent to \emph{score matching}, which is the basis of the score-based modeling techniques behind many of the recent advances in machine learning. 
\end{example}

\begin{example}[Continuous ranked probability score (CRPS)]\label{ex:CRPS2}
Let $\cY = \bbR$ and $\sP$ be the set of distributions with finite mean. The \emph{continuous ranked probability score} $\crps:\sP \to \overline{\bbR}$  given by $\crps(\bbP,y) = \int \left(F_\bbP(x) - \one\{y \le x\}\right)^2 \D x$,
where $F_\bbP$ is the CDF of $\bbP$, is a strictly proper scoring rule \citep{MathesonWinkler1976}. It has a compelling interpretation as the $L^{2}$ distance between the CDFs of $\bbP$ and the forecast $y$. Alternatively, it can also be seen as the Brier score for the binary event $\one{\{Y \leq x\}}$ integrated over $\bbR$. Furthermore, it has the equivalent representation
\begin{equation}\label{eq:CRPS_kernel}
\crps(\bbP,y) = \int |x - y| \D\bbP(x) -\frac{1}{2}\iint |x - x'| \D\bbP(x)\D\bbP(x'),
\end{equation}
see for example \citet[Lemma 2.1]{BaringhausFranz2004}. The representation in Equation \ref{eq:CRPS_kernel} identifies the CRPS as a \emph{kernel score}, see Section \ref{sec:kernel}. It is of great significance in meteorology, where it is the most popular scoring rule for real-valued variables.
\end{example}

While these examples consider rather standard outcome domains $\cY$, proper scoring rules have also been extended and applied to exotic domains such as distributions on the circle \citep{gneiting2007}, sphere \citep{takasu2018}, general Riemannian manifolds \citep{mardia2016}, functional data or stochastic processes \citep{Hayati2024}, point processes \citep{brehmer2024} and discrete or continuous trajectories \citep{bonnier2024}.

With the goal of forecast evaluation in mind, the notion of propriety at Equation \ref{eqn:propriety} is easily motivated, compare Equation \ref{eq:empirical}. However, the notion of propriety also appears naturally from the perspective of (distributional) regression, compare Equation \ref{eq:regression}. The following result has previously appeared in \citet{Tsyplakov2011,HolzmannEulert2014}. It is analogous to a result of \citet{banerjee2005} on conditional expectation and Bregman loss functions.

\begin{theorem}\label{thm:regression}
    Let $\cY$ be a Polish space and $\cX$ be some measurable covariate space. A scoring rule $S: \sP \times \cY \to \overline{\bbR}$ is (strictly) proper if and only if for every pair of random variables $(X,Y) \in \cX \times \cY$ such that the conditional distributions $\bbP_{Y |X=x}$ are in $\sP$, and every family $\{\bbP^{x}: x \in \cX\} \subseteq \sP$ \textit{\color{red} such that the map $(x, y) \mapsto S(\bbP^{x}, y)$ is measurable}, the expectation $\bbE \left[ S(\bbP^{X}, Y) \right]$ is minimized  if (and only if) $\bbP^{x} = \bbP_{Y | X = x}$ for $x \in \cX$ $\bbP_{X}$-almost surely, where $\bbP_X$ is the marginal distribution of $X$.
\end{theorem}

\begin{remark}
    Note that for the map $(x, y) \mapsto S(\bbP^{x}, y)$ to be measurable, it suffices that (a) the map $x \mapsto \bbP^{x}(A)$ is measurable, that is, $x \mapsto \bbP^{x}$ corresponds to a Markov kernel, and (b) the map $(\bbP, y) \mapsto S(\bbP, y)$ is measurable for every $y \in \cX$ in the Borel $\sigma$-algebra induced by the weak topology on $\sP$. Furthermore, if $\sP$ is Polish and the map $\bbP \mapsto S(\bbP, y)$ is continuous in the weak topology on $\sP$ for every $y \in \cX$ then together with the fact that $y \mapsto S(\bbP, y)$ is measurable for every $\bbP \in \sP$ (by definition), we can use Carath\'{e}odory's theorem \cite[see][Lemma 4.51]{AliprantisBorder2006}  to conclude that $(\bbP, y) \mapsto S(\bbP, y)$ is measurable (seems like a demanding requirement, it is probably better to require continuity in $y$).
\end{remark}

Mathematically, the concept of proper scoring rules is an extension of the concept of Bregman loss functions to probability distributions  \citep{Ovcharov2018}. These are distance-like functions on Euclidean spaces used in the optimization literature to adapt general methods such as proximal gradient descent to the geometry of specific problems \citep{Teboulle2018}. Importantly, many boosting algorithms such as \emph{AdaBoost} are known to be instances of one such method known as mirror descent \citep{Collins2002, Bejabattais2023}.

To every scoring rule $S: \sP \times \cY \to \overline{\bbR}$, we associate an \emph{entropy} $H: \sP \to \overline{\bbR}$ and a \emph{divergence} $d: \sP \times \sP \to \overline{\bbR}$ given by
\begin{equation}\label{eqn:entropy-div}
    H(\bbP) = S(\bbP, \bbP) = \int S(\bbP, y) \D\bbP(y) \qquad \mbox{ and } \qquad d(\bbP, \bbQ) = S(\bbP, \bbQ) - H(\bbQ).
\end{equation}
For a (strictly) proper scoring rule, the entropy $H$ is a (strictly) concave function on $\sP$, since, by propriety for all $\lambda \in (0,1)$,
\begin{multline*}
    H(\lambda \bbP + (1-\lambda)\bbQ) 
    = \lambda S(\lambda \bbP + (1-\lambda)\bbQ, \bbP) + (1-\lambda) S(\lambda \bbP + (1-\lambda)\bbQ,\bbQ) \\
    \ge \lambda S(\bbP,\bbP) + (1-\lambda) S(\bbQ,\bbQ) = \lambda H(\bbP) + (1-\lambda) H(\bbQ).
\end{multline*}

The divergence satisfies $d(\bbP, \bbQ) \geq 0$ for $\bbP, \bbQ \in \sP$ with equality if (and only if) $\bbP = \bbQ$. The divergence is generally not a metric on $\sP$ and, in fact, it does not even need to be symmetric in its arguments. However, for the important class of kernel scores, which includes the CRPS, symmetry holds and the square root of the divergence is a metric, see Example \ref{ex:CRPS_div} and Section \ref{sec:kernel}.

\begin{example}[Logarithmic score]\label{ex:log_div}
    The logarithmic score has entropy 
    \[H(\bbP) = -\int p(x) \log p(x) \D \mu(x),\] which is the Shannon entropy. The associated divergence $d(\bbP,\bbQ)  = \int q(y)\log (q(y)/p(y)) \D \mu(y) = D_{\mathrm{KL}}(\bbQ || \bbP)$ is the Kullback-Leibler divergence.
\end{example}

\begin{example}[Hyv\"{a}rinen score]\label{ex:div_hyv}
The Hyv\"arinen score has entropy 
\[H(\bbP) = -\tfrac{1}{2}\int \|\nabla\log p(\bx)\|^{2}\,p(\bx)\D\bx,\]
see Remark 1 in the appendix for a derivation.
The associated divergence $d(\bbP, \bbQ) = \tfrac{1}{2}\int \| \nabla \log p(\by) - \nabla \log q(\by) \|^{2}\, q(\by)\D\by$ is sometimes called score matching distance.   
\end{example}

\begin{example}[CRPS]\label{ex:CRPS_div}
The CRPS has entropy \[H(\bbP) = \int F_\bbP(x)(1-F_\bbP(x)) \D x\] and a symmetric divergence $d(\bbP,\bbQ) = \int (F_\bbP(y)-F_\bbQ(y))^2 \D y$, whose square root is a metric, compare Example \ref{ex:energy-scores}.
\end{example}

An alternative popular approach to comparing probability distributions is provided by the family of $\phi$-divergences (also known as $f$-divergences) such as the Hellinger and total variation distances. It should be noted that with the sole exception of the Kullback-Leibler divergence, these do not correspond to proper scoring rules \citep{Csiszar1991}, although they have been shown to be intimately related \citep{Mohamed2017,Gao2020}.

Proper scoring rules have been studied from functional analysis \citep{Ovcharov2015}, convex analysis \citep{Williamson2023} and information geometry \citep{dawid2007} perspectives.

\subsection{Characterization} 

Proper scoring rules can be elegantly characterized in terms of their entropy under minimal assumptions. 
A scoring rule $S: \sP \times \cY \to \overline{\bbR}$ is called \emph{regular} if $H(\bbP) = S(\bbP, \bbP)$ is finite and $S(\bbP, \bbQ) > -\infty$ for every $\bbP, \bbQ \in \sP$ with $\bbP \neq \bbQ$. 

\begin{theorem}[\citealp{gneiting2007}]\label{thm:gneiting-char}
    A regular scoring rule $S: \sP \times \cY \to \overline{\bbR}$ is (strictly) proper if and only if there is a (strictly) concave function $H: \sP \to \bbR$ such that
    \begin{equation}\label{eqn:gneiting-char}
        S(\bbP, y) = H(\bbP) + \langle h_{\bbP}, \delta_{y} - \bbP \rangle = H(\bbP) + h_{\bbP}(y) - \int h_{\bbP}(y) \D \bbP (y)
    \end{equation}
    for every $\bbP \in \sP$ and $y \in \cY$, where $h_{\bbP}$ is a supergradient of $H$ at $\bbP$.
\end{theorem}
In Equation \ref{eqn:gneiting-char}, we use the shorthand $\langle f,\bbP\rangle$ for the integral $\int f \D\bbP$ of a $\bbP$-integrable function $f$ on $\cY$, and $\delta_y$ denotes the Dirac measure at the point $y$. Furthermore, a supergradient of a concave function $H\colon \sP \to \overline{\bbR}$ at $\bbP$ is a function $h_\bbP\colon \cY \to \overline{\bbR}$ that is $\bbQ$-integrable and
\begin{equation}\label{eq:supergradient}
H(\bbP) + \langle h_\bbP, \bbQ - \bbP\rangle \ge H(\bbQ)
\end{equation}
for all $\bbQ \in \sP$. The function $H$ in Equation \ref{eqn:gneiting-char} indeed coincides with the entropy of $S$ defined in Equation \ref{eqn:entropy-div}. A regular proper scoring rule is thus determined by its entropy $H$ and the family $\{h_{\bbP}: \bbP \in \sP\}$ of chosen supergradients. Theorem \ref{thm:gneiting-char} is general but also somewhat abstract since supergradients of concave functions on convex sets of probability measures may not be directly intuitive. However, in the following, we will showcase the strength and direct applicability of this result in several instances. 

Theorem \ref{thm:gneiting-char} simplifies if the outcome space $\cY=\{y_1,\dots,y_n\}$ is finite with $n$ elements. Then, the probability measures $\bbP$ on $\cY$ can be identified with vectors $\bp = (p_{j})_{j=1}^{n}$ in the unit simplex $\sS_n = \{(p_{j})_{j=1}^{n} \in [0, 1]^{n} : \sum_{j=1}^{n} p_{j} = 1\}$, where $p_{j} = \bbP(\{y_j\})$. Theorem \ref{thm:gneiting-char} then states that regular proper scoring rules are all of the form
\begin{equation*}
        S(\bp, y_j) = H(\bp) + \langle h_{\bp}, \be_{j} - \bp \rangle,
\end{equation*}
where $H:\sS_n \to \bbR$ is a concave function with supergradient $h_{\bp} \in \bbR^n$ at $\bp$, $\be_{j} = (\delta_{ij})_{i=1}^{n}$, and $\langle\cdot,\cdot\rangle$ is the usual scalar product on $\bbR^n$. This characerization of proper scoring rules for discrete outcome spaces goes back to \citet{Savage1971} and \citet{mccarthy1956}. If $n=2$, the outcome is binary and the concave function $H$ can be defined on the unit interval $[0,1]$. Using a Choquet representation of univariate concave functions, the representation theorem of \citet{Schervish1989} for proper scoring rules for binary outcomes can be derived, compare \citet[Theorem 3]{gneiting2007}. 

Typically, the entropy $H$ is differentiable in the sense that for every $\bbP \in \sP$, there exists a $\sP$-integrable function $\nabla_{\bbP}H: \cY \to \overline{\bbR}$ such that
\begin{equation}\label{eq:diff_entropy}
    \lim_{\alpha \to 0^+} \frac{1}{\alpha}\Big[H((1-\alpha)\bbP + \alpha \bbQ) - H(\bbP)\Big] = \langle \nabla_\bbP H, \bbQ \rangle
\end{equation}
for every $\bbQ \in \sP$, and $\nabla_{\bbP}H$ is called the gradient of $H$. If $H$ is differentiable, then supergradients are unique and equal to the gradient  $\nabla_{\bbP}H = h_\bbP$ \citep[see][Proposition 5.3]{Ekeland1999}. Thus, for differentiable entropies, there is a bijective correspondence between a scoring rule and its entropy and Equation \eqref{eqn:gneiting-char} provides a generic recipe for constructing proper scoring rules out of differentiable concave functions on $\sP$. 

\subsection{Geometric Properties}\label{sec:geometry}

Generally, the divergence of a proper scoring rule does not correspond to a metric but behaves like a (squared) metric in many respects. Locally speaking, the divergence can be said to impose a manifold structure on the space of probability measures with the Hessian $\nabla_{\bbP}^{2}H$ of the entropy acting as the metric tensor by relating the ``distance" $d(\bbP, \bbP + \Delta\bbP)$ between two infinitesimally close points $\bbP, \bbQ \in \sP$ to the difference $\Delta \bbP = \bbQ - \bbP$ between their ``coordinates" $\bbP$ and $\bbP + \Delta \bbP$ as
\begin{equation}\label{eqn:local-metric}
    d(\bbP, \bbP + \Delta\bbP) \approx -\frac{1}{4} \langle \Delta\bbP, [\nabla_{\bbP}^{2}H] \Delta\bbP \rangle = -\frac{1}{4} \int \nabla_{\bbP}^{2}H(x,x')\D\Delta\bbP(x)\D\Delta\bbP(x'),
\end{equation}
see Theorem \ref{thm:div-hessian} for a precise statement. Thus, the divergence acts like a (squared) metric \emph{locally} between close enough points, with the metric depending on the points via the Hessian $\nabla_{\bbP}^{2}H$. This allows us to compare the local behavior of different scoring rules and has also been used to devise new scoring rules \citep[see][]{Bolin2023}. 

\begin{theorem}\label{thm:div-hessian}
    Let $H: \sP \to \overline{\bbR}$ be a twice-differentiable concave function and let $d: \sP \times \sP \to \overline{\bbR}$ denote the divergence of the corresponding scoring rule. Then 
    \begin{equation*}
        d(\bbP, \bbQ) = -\frac{1}{2} \iint \left[ \int_{0}^{1} t\nabla_{t\bbP + (1-t)\bbQ}^{2}H(x,x') \D t \right] \D(\bbQ - \bbP)(x)\D(\bbQ - \bbP)(x') 
    \end{equation*}
    for every $\bbP, \bbQ \in \sP$.
\end{theorem}
In other words, evaluating the divergence $d(\bbP, \bbQ)$ is a matter of integrating the Hessian over mixtures of $\bbP$ and $\bbQ$. This allows us to compare divergences of proper scoring rules to metrics which are more intuitive but less convenient in estimation. 

\begin{example}
Let $\bbP, \bbQ$ be absolutely continuous distributions on $\bbR^d$ with Lebesgue densities $p$ and $q$, respectively. If $\bbP$ and $\bbQ$ are ``close'', the divergence $d(\bbP, \bbQ)$ with respect to the logarithmic score, quadratic score (see Equation \ref{eq:Brier} in Section \ref{sec:f-scores}), and CRPS is approximately given by the right-hand side of Equation \ref{eqn:local-metric}, which is, respectively,
\begin{equation*}
    \frac{1}{4}\int (q(\bx)/p(\bx) - 1)^{2} p(\bfx)\D \bx, \; \frac{1}{2}\int (q(\bx) - p(\bx))^{2}  \D \bx, \; -\frac{1}{8}\iint |x-y|(q(\bx) - p(\bx))(q(\by) - p(\by))\D \bx\D \by.
\end{equation*}
\end{example}

The divergence also satisfies a generalization of the law of cosines (and the Pythagorean theorem). Indeed, the so-called three point property holds, which means that for $\bbP_{1}, \bbP_{2}, \bbP_{3} \in \sP$,
 \begin{equation*}
        d(\bbP_{1}, \bbP_{3}) = d(\bbP_{1}, \bbP_{2}) + d(\bbP_{2}, \bbP_{3}) - \langle \nabla_{\bbP_{2}}H - \nabla_{\bbP_{1}} H, \bbP_{3} - \bbP_{2}\rangle.
    \end{equation*}
 In particular, $d(\bbP_{1}, \bbP_{3}) = d(\bbP_{1}, \bbP_{2}) + d(\bbP_{2}, \bbP_{3})$ if and only if $(\bbP_{1}, \bbP_{2})$ and $(\bbP_{2}, \bbP_{3})$ are geodesically orthogonal, that is $\langle \nabla_{\bbP_{2}}H - \nabla_{\bbP_{1}} H, \bbP_{3} - \bbP_{2}\rangle = 0$.     
The geometry of proper scoring rules (like Bregman loss funcions) is said to be that of doubly flat manifolds, a class of Riemannian manifolds which captures the notion of ``flatness'' as exemplified by affine subspaces of Euclidean spaces \citep{dawid2007, amari2016}.

\section{Families of Scoring Rules}\label{sec:families}

\subsection{Kernel scores}\label{sec:kernel}

Kernel scores form a flexible and powerful class of proper scoring rules that generalize many familiar examples. They are based on reproducing kernel Hilbert space (RKHS) methods and offer a principled way to compare probability distributions using kernel-based embeddings. In what follows, we detail their construction and computational aspects.

\subsubsection{Construction and computation}
Kernel scores are a far-reaching generalization of the CRPS introduced in Example \ref{ex:CRPS}. They are derived from entropies of the form
\begin{equation}\label{eq:kernel_entropy}
H(\bbP) = \frac{1}{2}\iint h(x, x') \D\bbP(x) \D\bbP(x')
\end{equation}
for some conditionally negative definite kernel $h: \cY \times \cY \to [0,\infty)$, and have been popularized by \citet{dawid2007,gneiting2007}. 
A function $h: \cY \times \cY \to \bbR$ is a \emph{conditionally negative definite kernel}, if it is symmetric, that is, $h(x, y) = h(y, x)$ for $x, y \in \cY$, and for every $n \geq 1$, $x_1,\dots,x_n \in \cY$ and $\alpha_1,\dots,\alpha_n \in \bbR$ with $\sum_{j=1}^{n} \alpha_{j} = 0$, it satisfies
$\sum_{i, j=1}^{n} \alpha_{i}\alpha_{j} h(x_{i},  x_{j}) \leq 0$. Moreover, we assume that $h(x, y) \geq 0$ for $x, y \in \cY$. This does not lead to any loss of generality, see Remark 2 in the appendix.

Kernel scores are essentially squared distances between the kernel mean embeddings of the measures into some Hilbert space \citep{SteinwartZiegel2021}. They are closely related to maximum mean discrepancies (MMD), which were originally devised as test statistics for two-sample testing \citep{gretton2006}. Furthermore, they are intimately connected to energy statistics \citep{szekely2013}. The relationship between energy statistics and MMDs is carefully studied in \citet{SejdinovicSriperumbudurETAL2013}, see also \citet{lyons2013}.

For any conditionally negative definite kernel $h$, the entropy at Equation \ref{eq:kernel_entropy} is concave, and hence the induced scoring rule is proper. Strict propriety is a more delicate issue. Assume $h \ge 0$ and let $\sP_h$ be the set of all probability measures such that $H(\bbP) < \infty$. The kernel $h$ is \emph{strongly conditionally negative definite} if $\iint h(x, y) \D(\bbP-\bbQ)(x) \D(\bbP-\bbQ)(y) = 0$ implies $\bbP = \bbQ$ for any $\bbP, \bbQ \in \sP_{h}$. Strongly conditionally negative definite kernels are closely related to positive definite characteristic and universal kernels, see \citet{gretton2006,sriperumbudur2011}. 

\begin{theorem}\label{thm:kernel-score}
       If $h: \cY \times \cY \to [0,\infty)$ is a (strongly) conditionally negative definite kernel, then the scoring rule $S: \sP_h \times \cY \to \overline{\bbR}$ given by 
    \begin{equation}\label{eqn:kernel-score}
         S(\bbP, y) =  \int h(y, x) \D\bbP(x) -\frac{1}{2}\iint h(x, x') \D\bbP(x)\D\bbP(x')
    \end{equation}
    is (strictly) proper.
\end{theorem}

A \emph{kernel score} is a (strictly) proper scoring rule $S: \sP_h \times \cY \to \overline{\bbR}$ of the form in Equation \ref{eqn:kernel-score} for some (strongly) conditionally negative definite kernel $h: \cY \times \cY \to [0,\infty)$. 
The divergence $d$ of a kernel score is $d(\bbP, \bbQ) 
    = -\frac{1}{2} \iint h(y, y') \D(\bbP - \bbQ)(y) \D(\bbP - \bbQ)(y').$ 
We can write $S(\bbP, y) = d(\bbP, \delta_{y}) + \tfrac{1}{2}h(y, y)$ since $\delta_{y} \in \sP_h$ for $y \in \cY$. In fact, $(\bbP, y) \mapsto d(\bbP, y)$ is itself a proper scoring rule and is strongly equivalent to $S$. The intimate connection between conditionally negative definite kernels and the geometry of Hilbert spaces results in many interesting interpretations of the kernel score that are best stated in terms of its divergence. 

\begin{theorem}[Kernel Scores]\label{thm:kernel-scores-char}
    The function $S: \sP \times \cY \to \overline{\bbR}$ as in Equation \ref{eqn:kernel-score} is a (strictly) proper scoring rule if and only if there exists a Hilbert space $\sH$ and a subset $\{\psi_{x}\}_{x \in \cY} \subset \sH$ such that the divergence $d$ of $S$ satisfies
    \begin{equation}\label{eqn:kernel-score-embedding}
       d(\bbP, \bbQ) = -\frac{1}{2} \iint \|\psi_{x} - \psi_{y}\|_{\sH}^{2} \D(\bbP - \bbQ)(x) \D(\bbP - \bbQ)(y) = \left\|\int \psi_x \D (\bbP - \bbQ)(x)\right\|_{\sH}^2,
        \end{equation}
        (and that the mapping $\bbP \mapsto \int \psi_{x} \D\bbP(x)$ defined for $\bbP$ with $\int \|\psi_{x}\|^{2} \D\bbP(x) < \infty$ is injective).
\end{theorem}

We can thus construct a kernel score for a given outcome space $\cY$ as soon as we can construct a conditionally negative definite kernel on $\cY$ or a Hilbertian embedding of $\cY$. A simple approach is to choose $h = -k$ for some strictly positive definite kernel $k: \cY \times \cY \to \bbR$ which can be constructed in many ways \citep{williams2006}. For example, $k(\bx, \by) = -\exp(-\|\bx-\by\|^{p}/\lambda)$ corresponds to the Laplacian ($p=1$) and Gaussian ($p=2$) kernels on $\bbR^d$ leading to bounded kernel scores. Kernel scores implicitly correspond to different embeddings of the outcome space $\cY$ into a Hilbert space via Equation \ref{eqn:kernel-score-embedding}. For example, $(\bx, \by) \mapsto \|\bx - \by\|^{\beta/2}$ for $\beta \in (0, 2)$ is actually a metric on $\cY = \bbR^{d}$. According to a well-known result of \citet{schoenberg1937}, the corresponding metric space can be embedded into a Hilbert space, which implies the existence of an embedding $\{\psi_{\bx}\}_{\bx \in \cY} \subset \sH$ such that $\|\psi_{\bx} - \psi_{\by}\|_{\sH} = \|\bx - \by\|^{\beta/2}$ for $\bx, \by \in \cY$.

\begin{example}[Energy Scores]\label{ex:energy-scores}
    Let $\cY = \bbR^{d}$. For $\beta \in (0, 2)$, the \emph{energy score} $S_{\mathrm{Energy}}: \sP \times \cY \to \overline{\bbR}$ given by
    \begin{equation*}
        S_{\mathrm{Energy}}(\bbP, \by) = \int \| \bx - \by\|^{\beta}\D\bbP(\bx) -\frac{1}{2}\iint\|\bx - \bx'\|^{\beta}\D\bbP(\bx)\bbP(\bx') 
    \end{equation*}
   is strictly proper for $\sP = \{\bbP: \int \|\bx\|^{\beta} \D\bbP(\bx) < \infty\}$. In fact, the same applies if we replace $\|\cdot\|$ with $\|\cdot\|_{\alpha}$ given by $\|\bx\|_{\alpha} = (\sum_{j=1}^{d}x_{j}^{\alpha})^{1/\alpha}$ where $\bx = (x_{j})_{j=1}^{d} \in \bbR^d$, for $\alpha \in (0, 2)$, $\beta \in (0, \alpha]$ and the resulting scores are called \emph{non-Euclidean energy scores} \citep{gneiting2007}. The term energy score is used, since they resemble potential energy functions of central forces in physics \citep{szekely2013}. Unlike kernel scores with Gaussian or Laplacian kernels, energy scores are \emph{homogeneous} or \emph{scale-free}, that is, replacing $\bx$, $\bx'$, and $\by$ by $c\bx$, $c\bx'$, and $c\by$ respectively for some $c > 0$ only scales the score by $c^{\beta}$.
    The most popular kernel score is the CRPS, which is a special case of an energy score with $d = 1$ and $\beta = 1$, compare Examples \ref{ex:CRPS2}, \ref{ex:CRPS_div} and Section \ref{sec:crps}.
\end{example}

\begin{example}[Variogram Score]
Let $\cY = \bbR^{d}$.
    \citet{Scheuerer2015} propose the \emph{variogram score}
    \begin{equation*}
        S(\bbP, \by) = \sum_{i, j =1}^{d} w_{ij} \left(|y_{i} - y_{j}|^{p} - \int |x_{i} - x_{j}|^{p}\D\bbP(\bx) \right)^{2}
    \end{equation*}
    where $w_{ij} \geq 0$, $\bx = (x_{i})_{i=1}^{d}$, $\by = (y_{i})_{i=1}^{d}$, and $p \in (0, \infty)$, defined on $\sP = \{\bbP : \int |x_i|^p \D \bbP(\bx) < \infty, i=1,\dots,d\}$. 
    It is a proper kernel score with conditionally negative definite kernel $h(x,y) = \sum_{i,j=1}^d w_{ij} (|x_i-x_j|^p - |y_i - y_j|^p)^2$, see \citet{allen2023a}. However, it is not strictly proper. 
\end{example}

An advantageous feature of kernel scores is that they can be easily computed for discrete predictive distributions, or, generally, estimated by sampling from $\bbP$. Indeed, for an iid sample $X_1,\dots,X_m$ from $\bbP$,
\begin{equation}\label{eqn:kernel-sampling}
    \hat{S}(\bbP, y) = \frac{1}{m}\sum_{i=1}^{m} h(X_{i}, y) - \frac{1}{2m(m-1)}\sum_{i =1}^{m}\sum_{i': i' \neq i} h(X_{i}, X_{i'})
\end{equation}
is an unbiased estimate of $\hat{S}(\bbP, y)$. This provides an affordable alternative to numerical integration when a closed-form expression for $S(\bbP, y)$ is not available.

\subsubsection{Characterizations}

The divergence of a proper scoring rule is generally not symmetric, although locally it behaves like a squared metric, as discussed in Section \ref{sec:geometry}. Interestingly, the square root of the divergence of a kernel score is always a valid metric on $\sP$. In fact, under mild conditions, these are the only proper scoring rules with this property. 

\begin{theorem}\label{thm:sym-char}
    Let $S: \sP \times \cY \to \overline{\bbR}$ be a regular proper scoring rule such that $\{\delta_{y}: y \in \cY\} \subseteq \sP$ and the supergradient map $\bbP \mapsto h_{\bbP}$ is weakly continuous. Then, the divergence $d$ of $S$ is symmetric if and only if $S$ is strongly equivalent to a kernel score. 
\end{theorem}

The definition of weak continuity of the supergradient map is given in the appendix. Theorem \ref{thm:sym-char} need not be true when $\{\delta_{x}: x \in X\} \not\subseteq \sP$ as exemplified by the quadratic score for Lebesgue densities on $\bbR^d$, see Section \ref{sec:f-scores}. It has divergence $d(\bbP, \bbQ) = \int (p(\by) - q(\by))^{2} \D \by$ but is not a kernel score, see Remark 3 in the appendix.

We focus now on the outcome space $\cY = \bbR^d$. If we do not want to attach special importance to particular regions of $\cY = \bbR^{d}$, the scoring rule should be \emph{translation invariant}, that is, $S(\bbP, \by) = S(\bbP_{\bh}, \by + \bh)$ for $\by, \bh \in \bbR^{d}$, where $\bbP_{\bh}(A) = \bbP(A + \bh)$ for Borel sets $A \subset \bbR^{d}$. Such kernel scores correspond to kernels of the form $h(\bx, \by) \equiv h(\bx - \by)$, $\bx, \by \in \bbR^{d}$, for some conditionally negative definite function $h: \bbR^{d} \to \bbR$. As a consequence of the Levy-Khinchin Theorem \citep[see][Theorem 3.12.12]{sasvari2013}, they admit an elegant representation as the $L^{2}(\mu)$-distance of the characteristic functions, where $\mu$ is the spectral measure. This relates to the results of \citet{SzekelyRizzoETAL2007}. Similar results have independently been shown by \citet{modeste2024} as the characterization of translation invariant MMDs on $\bbR^{d}$.

\begin{theorem}[Spectral Representation of Translation Invariant Kernel Scores] \label{thm:spectral-rep-kernel}
    Let $S: \sP \times \bbR^d \to \overline{\bbR}$ be a translation invariant kernel score. There exists a positive semi-definite matrix $\bB \in \bbR^{d \times d}$, and a $\sigma$-finite measure $\mu$ on $\bbR^{d}$ with $\mu(\{\bzero\}) = 0$ and $\int_{\bbR^{d}} \min(1, \|\bu\|^{2}) \D\mu(\bu) < \infty$ such that
    \begin{alignat*}{4}
         S(\bbP, \by) &= \textstyle \quad &&(\by - \mathbf{m}_{\bbP})^{\top}\bB (\by - \mathbf{m}_{\bbP}) \quad&&+\quad &&\int_{\bbR^{d}} |e^{i\bu^{\top}\by} - f_{\bbP}(\bu)|^{2} \D\mu(\bu) \\
     d(\bbP, \bbQ) &= \textstyle \quad &&(\mathbf{m}_{\bbQ} - \mathbf{m}_{\bbP})^{\top}\bB (\mathbf{m}_{\bbQ} - \mathbf{m}_{\bbP}) \quad&&+\quad &&\int_{\bbR^{d}} |f_{\bbQ}(\bu)-f_{\bbP}(\bu)|^{2} \D\mu(\bu)
    \end{alignat*}
     for every $\by \in \cY$, and $\bbP, \bbQ \in \sP$ with characteristic functions $f_\bbP$, $f_\bbQ$, and means $\mathbf{m}_\bbP$, $\mathbf{m}_\bbQ$, respectively. In particular, if the negative kernel is a positive definite function, then $\mu$ is a finite measure. 
\end{theorem}

It follows that different translation invariant kernel scores place different weight on the different frequencies of the density. For example,  the Gaussian and Laplacian kernel scores for $\lambda > 0$, and the energy score for $\beta \in (0, 2)$ in $\bbR^{d}$, correspond to $\bB = \bzero$ and the spectral measures $\mu$ are given by 
\begin{equation*}
    \frac{\mathrm{d}\mu}{\mathrm{d}\bu} \,\,\propto\,\, \exp(-\lambda\|\bu\|^{2}), \quad \Big(1+\lambda^{2}\|\bu\|^{2}\Big)^{-(d+1)/2}, \mbox{ and }  \|\bu\|^{-(d+\beta)},
\end{equation*}
respectively. 

Further important notions of equivariance and invariance are homogeneity and isotropy. A proper scoring rule $S: \sP \times \bbR^{d} \to \overline{\bbR}$ is said to be \emph{homogeneous} of degree $\alpha$ if $S(\bbP_{c}, c\by) = c^{\alpha}S(\bbP, \by)$ for every $c > 0$, $\bbP \in \sP$, and $\by \in \bbR^{d}$ where $\bbP_{c}(A) = \bbP(c^{-1}A)$ for Borel sets $A \subset \bbR^{d}$. It is said to be \emph{isotropic} or \emph{rotation-invariant} if $S(\bbP_{\bU}, \bU\by) = S(\bbP, \by)$ for every rotation matrix $\bU \in \mathrm{SO}(d)$, $\bbP \in \sP$, and $\by \in \bbR^{d}$ where $\bbP_{\bU}(A) = \bbP(\bU^\top A)$ for Borel sets $A \subset \bbR^{d}$. A consequence of homogeneity, that is of great practical importance is that scaling the sample $\by$ by $c > 0$ as $c\by$, scales the predictive distribution accordingly. Thus, using different units of measurement for the data (e.g. km/h and knots for wind speed) leads to the same optimal distribution in the corresponding units (in estimation contexts), or the same predictive performance assessment (in forecast evaluation contexts). The following result characterizes translation invariant kernel scores with these properties.

\begin{theorem}\label{thm:kernel-scores-homogeneous-iso}
    Let $S: \sP \times \bbR^{d} \to \overline{\bbR}$ be a translation invariant kernel score with the matrix $\bB$ and the spectral measure $\mu$ as in Theorem \ref{thm:spectral-rep-kernel}. Then there exists a $\sigma$-finite measure $\rho$ on $(0, \infty)$ with $\int_{0}^{\infty} \min(1, r^{2})\D\rho(r) < \infty$ and probability measures $\{\nu_{r}: r \in (0, \infty)\}$ on the unit sphere $\mathbb{S}^{d-1} \subset \bbR^{d}$ such that
    \begin{equation*}
        \D\mu(r \sigma) = \D\nu_{r}(\sigma) \D\rho(r)
    \end{equation*}
    for $r \in (0, \infty)$ and $\sigma \in \mathbb{S}^{d-1}$.
    \begin{enumerate}
        \item If $S$ is homogeneous of degree $\alpha \in \bbR$, then either (a) $\alpha = 2$ and $\mu$ is zero everywhere, or (b) $\alpha \in (0, 2)$, $\bB = \bzero$ and $\mu$ satisfies
    \begin{equation*}
        \D\mu(r\sigma) = \alpha \mu(\mathbb{S}^{d-1}\times(0,1))\frac{\D\nu(\sigma) \D r}{r^{1+\alpha}} 
    \end{equation*}
    for some probability measure $\nu$ on $\mathbb{S}^{d-1}$ and the Lebesgue measure $\mathrm{d}r$ on $(0, \infty)$.
    
    \item If $S$ is isotropic, then $\bB = c\mathbf{I}$ for some $c \geq 0$ and $\mu$ satisfies
    \begin{equation*}
        \D\mu(r\sigma) = \frac{\Gamma(d/2)}{2\pi^{d/2}}\, \D\sigma\, \D\rho(r)
    \end{equation*}
    where $\mathrm{d}\sigma$ is the (probability) Haar measure on $\mathbb{S}^{d-1}$. 
    \end{enumerate}
    In particular, the CRPS is the only $1$-homogeneous translation invariant kernel score on $\bbR$ and the energy scores are the only homogeneous isotropic translation invariant kernel scores on $\bbR^{d}$ (up to positive multiplicative constants).
\end{theorem}

\subsubsection{Continuous ranked probability score (CRPS)}\label{sec:crps}

The CRPS is arguably the most popular kernel score for univariate data. It was originally derived as the limit of discrete ranked probability scores and by integrating the Brier score, until it was recognized as a member of a much larger family of kernel scores \citep{Holstein1977, gneiting2007}. Mathematically, it is essentially the only $1$-homogeneous translation invariant kernel score on $\bbR$ (see Theorem \ref{thm:kernel-scores-homogeneous-iso}). In addition to historical reasons, its popularity is explained by a range of advantages that it offers.

The CRPS is an intuitive, interpretable, and convenient measure of forecasting accuracy.
It is distance-sensitive, that is, it assigns better scores to forecasts, which put more probability closer to the observations. In contrast, the logarithmic score assigns harsh penalties regardless of the overall closeness of the forecast to the observations if the forecast density at any of the observations is too low. 
The CRPS is essentially the $L^{2}$ distance between the CDFs of the forecast and the observations (see Example \ref{ex:CRPS_div}), and as the integral of the quadratic (Brier) score for binary outcomes, it elegantly connects to the conceptually simpler point forecasting setting. Indeed, it reduces to the absolute error loss when point forecasts are issued. Moreover, homogeneity implies that scaling the observations scales the predictive distribution accordingly even when the CRPS is away from the minimum.
Furthermore, the CRPS exists for all forecasts with a finite first moment, which means that the forecast measure need not necessarily have a density. It can be easily estimated by Equation \ref{eqn:kernel-sampling} using samples drawn from the forecast, or explicitly computed if the forecast distribution has finite support. This is especially useful in applications such as ensemble forecasting in meteorology, where forecasts are routinely issued as discrete predictive distributions consisting of the empirical distributions of the ensemble of predictions.  

The CRPS has been extended to circular random variables \citep{gneiting2007}. It has also spawned numerous other modifications, see Section \ref{sec:stat-performance}.

Closed-form expressions for the CRPS have been obtained for many important parametric families such as the Gaussian and generalized extreme value (GEV) distributions \citep{Taillardat2016}. Closed-form expressions are also available for forecast CDFs which are linear splines \citep{GasthausBenidisETAL2019}. In the absence of closed-form expressions, one can resort to estimating the score by sampling as in Equation \ref{eqn:kernel-sampling} or numerical integration techniques such as kernel quadrature \citep{Adachi2025}. Directly implementing Equation \ref{eqn:kernel-sampling} for the CRPS leads to an inefficient $\mathcal{O}(n^{2})$ time computation where $n$ is the number of samples drawn from the forecast. This can be improved to almost linear $\mathcal{O}(n\log n)$ time using an algebraically equivalent expression \citep{Kelen2025}:
\begin{equation*}
    \mathrm{CRPS}(\hat{\bbP}, y) = \frac{2}{n(n-1)}\sum_{j=1}^{n}(X_{(j)}-y)\left((n-1)\mathbf{1}_{\{y < X_{(j)}\}} - j + 1\right)
\end{equation*}
where $\hat{\bbP} = \frac{1}{n} \sum_{j=1}^{n} \delta_{X_{j}}$.

The CRPS plays a central role in forecast verification in meteorology. Moreover, with the widespread adoption of weather-related renewable energy sources such as solar and wind, accurate probabilistic forecasts for power, load and price of energy are necessary due to their highly variable nature \citep{Zhang2014, Hong2016, Weron2014}. In these applications, the CRPS is widely used as an evaluation metric and increasingly also for estimation. Furthermore, the CRPS has also been used for parametric \citep{GneitingRafteryETAL2005} and nonparametric \citep{HenziZiegelETAL2021} regression in statistics, and for supervised learning of distributions in machine learning \citep{GasthausBenidisETAL2019}. We discuss the role and use of proper scoring rules in estimation further in Section \ref{sec:applications}. 

\citet{Pic2023} derived the minimax rate of convergence for distributional regression with the CRPS and showed that it is achieved by using $k$-nearest neighbors ($k$-NN) or kernel smoothing. Surprisingly, this rate is of the same order as the minimax optimal rate for point prediction. Theoretical guarantees for estimation, model selection and convex aggregation have been obtained by \citet{Dombry2024}.

\subsection{Local scoring rules}\label{sec:local}

A local scoring rule is a scoring rule whose value $S(\bbP, x)$ depends on $\bbP$ only through the values of its density on an arbitrarily small neighborhood around $x$. Let $\cY \subseteq \bbR^d$ be open and let $\sP_{\mathrm{ac}}^{k}$ denote the set of probability measures on $\cY$ with Lebesgue densities, which are $k$-times differentiable.
 
\begin{definition}
    A \emph{local scoring rule of order $k$} is a proper scoring rule $S: \sP_{\mathrm{ac}}^{k} \times \cY \to \overline{\bbR}$ of the form
    \begin{equation*}
        S(\bbP, \by) = s(\by, p(\by), \nabla_{\by}\, p(\by), \dots, \nabla_{\by}^{k}\, p(\by))
    \end{equation*}
    where $s: \cY \times \bbR \times \cdots \times \bbR^{d^{k}} \to \bbR$, $p$ is the density of $\bbP$ and $\nabla_{\by}^{j}\, p(\by)$ denotes the $j$th derivative of $p$. 
\end{definition}

Arguably, the most popular and important of all proper scoring rules is the logarithmic score given by $S_{\mathrm{log}}(\bbP, \by) = -\log p(\by)$, compare Examples \ref{ex:logarithmic2} and \ref{ex:log_div}. It is a local scoring rule of order $0$, and remarkably, the only differentiable scoring rule of this kind \citep{bernardo1979}. 
Empirical risk minimization for the logarithmic score is equivalent to maximum likelihood estimation which is the foundation of classical parametric inference. 
The unparalleled historical success of maximum likelihood is in great part because it furnishes closed-form estimates for many important parametric families of distributions and is asymptotically optimal under rather mild regularity conditions. 

\begin{remark}
    Theorem \ref{thm:gneiting-char} can be used to derive a simple proof that the logarithmic score is the only differentiable local scoring rule of order $0$. Indeed, let $S(\bbP, \by) = s(\by, p(\by))$ be such a scoring function. The function $s(\by, z)$ is differentiable in $z$, and its entropy is $H(\bbP) = \int s(\bx,p(\bx))p(\bx) \D\bx$. 
    By Equation \eqref{eqn:gneiting-char} and using Equation \eqref{eq:diff_entropy}, we must have
    \begin{equation*}
        S(\bbP, \by) = s(\by, p(\by)) + p(\by) \frac{\partial}{\partial z}s(\by, p(\by)) - \int \frac{\partial }{\partial z}s(\bx, p(\bx)) [p(\bx)]^{2}  \D\bx,
    \end{equation*}
    which is of the form $s(\by, p(\by))$ if and only if the integral is a constant, or equivalently, the integrand
    is a constant multiple of $p(\by)$. It follows that $s(\by, p(\by)) = c \log p(\by) + d$ for some constants $c, d \in \bbR$.
\end{remark}

Even though the logarithmic score is the preferred approach to fitting standard parametric models, it is frequently impractical for fitting non-standard ones because the normalizing constants of such families are often not available in closed form and estimating or approximating them is computationally infeasible. To circumvent this problem, \citet{hyvarinen2005} proposed to fit models using \emph{score matching}, which is to minimize the expected squared distance $\bbE_{\bX \sim p_{\theta}}[\| \nabla_{\bx} \log p_{\theta}(\bX) - \nabla_{\bx} \log q(\bX) \|^{2}]$ between the ``scores'' $\nabla_{\bx} \log p_{\theta}(\bx)$ and $\nabla_{\bx} \log q(\bx)$, where $q$ is the true density. 
Empirically, this corresponds to minimizing the Hyv\"{a}rinen score given by \[S_{\mathrm{Hyv\ddot{a}rinen}}(\bbP, \by) = \Delta_{\by}\log p(\by) + \tfrac{1}{2}\|\nabla_{\by} \log p(\by)\|^{2},\] where $p$ is the density of $\bbP$, see Examples \ref{ex:hyvarinnen} and \ref{ex:div_hyv}. It is a local scoring rule of order $2$. The value of the Hyv\"arinen score does not change if we replace $p$ by $c \cdot p$ for some constant $c > 0$, which implies that the score can be evaluated even with non-normalized $p$.

The Hyv\"{a}rinnen score has also been extended to distributions on $\{-1, 1\}^{d}$, $\bbR_{+}^{n}$ \citep{hyvarinen2007}, spheres \citep{takasu2018} and general Riemannian manifolds \citep{dawidlauritzen2005, mardia2016}. Due to its applicability to non-normalized models, score matching enjoys significant popularity in the machine learning community and has, in fact, played an important role in recent advances in generative modeling, see Section \ref{sec:nonnorm}.  

\citet{Forbes2015} establish the asymptotic normality of parameter estimation using the Hyv\"{a}rinnen score and \citet{Koehler2022} relate the asymptotic efficiency of the estimator to the isoperimetric constant of the distribution. \citet{Ghosh2025} establish minimax rate optimality (up to log factors) of denoising. 

A fairly general characterization of local scoring rules of arbitrary order has been obtained by \citet{parry2012} using a calculus of variations approach. In particular, there are no local scoring rules of odd orders, and all local scoring rules of even order $\ge 2$ can be computed without knowledge of the normalization constant of the density like in the case of the Hyv\"arinen score. A more explicit characterization for the particular case of order $2$ was derived by \citet{ehm2012} using an approach based on Theorem \ref{thm:gneiting-char}. For example, \citet[Example 5.3]{ehm2012} introduce the log-cosh score 
\[
S_{\text{log-cosh}}(\bbP,y) = -\log\cosh z_1 + z_1 \tanh z_1 + z_2 (1-\tanh^2 z_1),
\]
with $z_1 = (\mathrm{d}/\mathrm{d}y)\log p(y) = p'(y)/p(y)$ and $z_2 = (\mathrm{d}^2/\mathrm{d}y^2)\log p(y) = p''(y)/p(y) - (p'(y)/p(y))^2$. It is strictly proper for distributions $\bbP$ on $\bbR$ with densities $p$ that satisfy some regularity conditions \citep[Definition 2.3]{ehm2012}. \citet{parry2016-ext} studied local scoring rules in the multivariate setting. Specifically, he introduced and characterized the so-called extensive scoring rules, which are multivariate scoring rules that decompose into sums of univariate scoring rules for individual components.
\citet{dawid2012} characterized scoring rules for general discrete outcome spaces which are local in that they depend only on the values of the probability mass function on a neighborhood of $x$ in a given undirected graph.

\subsection{$f$-Scores, $gf$-scores and generalized kernel scores}\label{sec:f-scores}

Let $\mu$ be a $\sigma$-finite measure on $\cY$ and $\sP_{\mathrm{ac}}$ the set of measures which are absolutely continuous with respect to $\mu$. For example, if $\cY = \bbR^d$ and $\mu$ is the Lebesgue measure, we consider absolutely continuous distributions on $\bbR^d$; if $\cY$ is finite and $\mu$ is the counting measure, we have categorical random variables.

Given a concave and differentiable function $f: [0,\infty)\to \bbR$, we can define a concave entropy function $H_{f}: \sP \to \overline{\bbR}$ as 
\begin{equation}\label{eq:f-entropy}
H_f(\bbP) = \int f(p(x)) \D\mu(x),
\end{equation}
where $p$ denotes the density of $\bbP \in \sP \subseteq \sP_{\mathrm{ac}}$ such that $H_f(\bbP) < \infty$ for $\bbP \in \sP$. Using Theorem \ref{thm:gneiting-char} and Equation \eqref{eq:diff_entropy}, it follows that the corresponding proper scoring rule $S_{f}: \cX \times \sP \to \overline{\bbR}$ is given by 
\begin{equation*}
    S_{f}(\bbP,y) = f'(p(y)) + H_f(\bbP) - \int f'(p(x)) p(x) \D \mu(y).
\end{equation*}
Such scoring rules were introduced by \citet{grunwald2004} as separable Bregman scores, although we refer to them as \emph{$f$-scores} to avoid confusion with Bregman loss functions and divergences. \citet{Perez1967} introduced the entropies in Equation \ref{eq:f-entropy} as generalized $f$-entropies. However, they are different from the $f$-entropies of \citet{Csiszar1972}. $f$-scores are the only proper scoring rules of the form $S(\bbP, y) = g(p(y)) - l(\bbP)$ \citep{dawid2007}. 

Many important scores are instances of this family. The logarithmic score arises for $f(u) = -u\log u$, compare Example \ref{ex:logarithmic2} and Section \ref{sec:local}. The concave function $f(u) = -u^2$ leads to the well-known \emph{Brier} or \emph{quadratic score} 
\begin{equation}\label{eq:Brier}
S_{\mathrm{Brier}}(\bbP, y) = -2p(y) + \int p(x)^{2} \D\mu(x), 
\end{equation}
compare Example \ref{ex:Brier}, where it is stated in an equivalent form for binary outcomes.

An extension of $f$-scores are $gf$-scores with entropy functions of the form 
\begin{equation}\label{eq:fg-entropy}
H_{gf}(\bbP) = g\Big(\int f(p(x)) \D\mu (x)\Big) = g(H_f(\bbP)),
\end{equation}
where $f$ and $g$ are such that $H$ is concave. For example, $f$ could be concave and $g$ concave and increasing, or, $f$ convex and $g$ concave and decreasing. If $f$ and $g$ are differentiable, applying Theorem \ref{thm:gneiting-char} and Equation \eqref{eq:diff_entropy} yields 
\begin{equation*}
    S_{gf}(\bbP,y) = g'(H_f(\bbP))f'(p(y)) + g(H_f(\bbP)) - g'(H_f(\bbP))\int f'(p(x)) p(x) \D\mu(x).
\end{equation*}
The spherical \citep{Roby1965} and pseudospherical scoring rules \citep{Good1971} given by 
\begin{equation*}
    S_{\mathrm{Spherical}}(\bbP, y) = -\frac{p(y)}{\left[\int p(x)^{2} \D\mu(x)\right]^{1/2}} \quad \text{and} \quad S_{\mathrm{pSpherical}}(\bbP, y) = -\frac{p(y)^{\alpha - 1}}{\left[\int p(y)^{\alpha} \D\mu(x)\right]^{1-1/\alpha}},
\end{equation*}
respectively are instances of $gf$-scores with $f(u) = u^{\alpha}$ and $g(u) = -u^{1/\alpha}$ for $\alpha = 2$ and $\alpha > 1$, respectively. In this case, strict concavity of the entropy follows from the Minkowski inequality.

The logarithmic, quadratic and spherical scores have been axiomatically characterized by \citet{Shuford1966}, \citet{Selten1998} and \citet{Jose2009}, respectively. Recently, \citet{Shao2024} explored their application to training large language models (LLMs).

Another twist on the idea is to replace the integral in Equation \ref{eq:fg-entropy} with a different entropy, for example, the entropy of a kernel score as at Equation \ref{eq:kernel_entropy}, that is, considering entropies of the form 
\[
H(\bbP) = g\left(\frac{1}{2}\iint h(x,x') \D\bbP(x)\D\bbP(x')\right) = g(H_{[h]}(\bbP))
\]
for a conditionally negative definite kernel $h$ and a differentiable increasing concave function $g:[0,\infty) \to \bbR$. This leads to the \emph{generalized kernel scores} given by 
\begin{equation*}
    S_{g[h]}(\bbP, y) =  g'(H_{[h]}(\bbP))\int h(x,y)\D\bbP(x) +  g(H_{[h]}(\bbP)) - g'(H_{[h]}(\bbP))H_{[h]}(\bbP),
\end{equation*}
introduced by \citet{Bolin2023}.

\subsection{Proper scoring rules induced by elicitable functionals}

Proper scoring rules can also be constructed by reducing the probability measure $\bbP$ first to some summary measure, a functional $T:\sP \to \cZ$ and then comparing distributions by comparing their summaries. Such scoring rules are typically proper but not strictly proper. 

A functional $T:\sP \to \cZ$ is called \emph{elicitable} if there exists a scoring function $s:\cZ\times \cY \to \overline{\bbR}$ such that
\[
\bbE s(T(\bbP),Y) \le \bbE s(z,Y)
\]
for all $z \in \cZ$, $\bbP \in \sP$, where $Y \sim \bbP$, subject to the existence of expectations, with equality if and only if $z = T(\bbP)$. Such a scoring function $s$ is called \emph{strictly consistent} for the functional $T$. Relevant functionals that are elicitable are the mean, quantiles, expectiles, and ratios of expectations. Some important functionals like the variance are not elicitable. There is a large literature on elicitability with possible starting points being \citet{LambertPennockETAL2008,Gneiting2011,FisslerZiegel2016}.

Given a strictly consistent scoring function $s$ of an elicitable functional $T$, the scoring rule $S:\sP \times \cY \to \overline{\bbR}$ given by
\[
S(\bbP,y) = s(T(\bbP),y)
\]
is proper, see \citet[Theorem 3]{Gneiting2011}. For example, since the squared error $s(x,y) = (x-y)^2$ is strictly consistent for the mean, the scoring function $(\bbP,y) \mapsto (\int x\D\bbP(x) - y)^2$ is proper but it is clearly not strictly proper unless the class $\sP$ is a location family. 

\begin{example}[Dawid-Sebastiani Scores] 
The pair of expectation and covariance matrix of multivariate distributions is elicitable. The resulting proper scoring rules have been introduced by \citet{dawid1999}.  An interesting member of this class is given by 
 \begin{equation*}
        S_{\mathrm{DS}}(\bbP, \by) = \log \det \Sigma_{\bbP} + (\by - m_{\bbP})^{\top}\Sigma_{\bbP}^{-1}(\by - m_{\bbP}),
    \end{equation*}
where $\by \in \bbR^d$ and $m_{\bbP} = \bbE[\bX]$ and $\Sigma_{\bbP} = \bbE[(\bX - m_{\bbP})(\bX - m_{\bbP})^{\top}]$ are the mean and covariance of $\bX \sim \bbP$. It is assumed that the covariance matrix is non-singular. The proper scoring rule $S_{\mathrm{DS}}$ is equivalent to the negative log-likelihood of a multivariate Gaussian distribution. As in the univariate case, it is a proper scoring rule for all distributions with non-singular covariance matrix. 
\end{example}

\section{Scoring Rule Decompositions}\label{sec:decomposition}

Proper scoring rules can be decomposed into interpretable components that provide insight concerning the performance of a given method. In the context of estimation, it is natural to consider an analogue of a (conditional) bias-variance decomposition, whereas in forecast evaluation, proper scoring rules are typically decomposed into components quantifying miscalibration, discrimination ability, and uncertainty. 

The familiar bias-variance decomposition of the squared error loss can be generalized to Bregman divergences \citep{pfau2013, Buja2005}. 
Since proper scoring rules are extensions of Bregman loss functions, we can derive a similar decomposition of the performance of a probabilistic predictor. Following the conditional setting in Section \ref{sec:psr-estimation}, let $\smash{\bbP_{\hat{\theta}, x}}$ be an estimator of $\bbP_{Y|X = x}$ where the parameter $\smash{\hat{\theta}} = \smash{\hat{\theta}(\cD)}$ has been estimated from a sample $\cD$ of $(X, Y)$. 
Then we can write
\begin{equation*}
    \bbE[d(\smash{\bbP_{\hat{\theta}, X}}, \bbP_{Y|X})] = \underbrace{\bbE[d(\smash{\bbP_{\hat{\theta}, X}}, \Bar{\bbP}^{\phantom{\ast}}_{X})]}_{\mathrm{variance}} + \underbrace{\bbE[d(\Bar{\bbP}^{\phantom{\ast}}_{X}, \bbP_{Y|X})]}_{\mathrm{bias}}
\end{equation*}
where $\Bar{\bbP}_{x} = \argmin_{\bbP} \bbE[d(\smash{\bbP_{\hat{\theta}, x}}, \bbP)]$ and $\bbP_{Y|X =x} = \argmin_{\bbP} \bbE[S(\bbP, Y) \,|\, X = x]$ is the best predictor, assuming the two exist. If we think of $\smash{\Bar{\bbP}_{x}}$ as the ``expectation'' of the predictor $\smash{\bbP_{\hat{\theta}, x}}$ over $\cD$, then just like in the original bias-variance decomposition, the variance term measures how much the estimator $\smash{\bbP_{\hat{\theta}, x}}$ varies around $\smash{\Bar{\bbP}_{x}}$ and the bias term measures the distance between $\smash{\Bar{\bbP}_{x}}$ and the best predictor $\bbP_{Y|X=x}$. Indeed, the two decompositions coincide for the Brier score. While $\Bar{\bbP}_{x}$ corresponds to a \emph{left Bregman projection}, commonly studied in optimization, the optimal predictor $\bbP_{Y|X=x}$ is a \emph{right Bregman projection}, which is of greater importance in statistics \citep{Bauschke2018}. 

In meteorology, a different kind of decomposition is used for assessing the overall performance of probabilistic forecasts, see \citet{Broecker2009,Kull2015,ArnoldWalzETAL2024} and references therein. The expected proper scoring rule is decomposed into terms corresponding to discrimination ability (DSC, or resolution), miscalibration (MCB, or reliability) and uncertainty (UNC). Let $\hat{\bbP}$ denote a (random) probabilistic forecast of $Y$. Then we can write
\begin{equation*}
    \bbE[S(\hat{\bbP}, Y)] 
    = \underbrace{\,\bbE[d(\hat{\bbP}, \bbP_{\smash{\raisebox{-1.0ex}{$\hat{}\,$}}})]\,}_{\mathrm{MCB}}
    - \underbrace{\,\bbE[d(\Bar{\bbP}, \bbP_{\smash{\raisebox{-1.0ex}{$\hat{}\,$}}})]\,}_{\mathrm{DSC}} 
    + \underbrace{\,H(\Bar{\bbP}_{\phantom{\smash{\raisebox{-0.9ex}{$\hat{}\,$}}}})}_{\mathrm{UNC}\,}
\end{equation*}
where $\bbP_{\smash{\raisebox{-1.2ex}{$\hat{}\,$}}}$ denotes the conditional law of the outcome $Y$ given the forecast $\hat{\bbP}$ and $\Bar{\bbP}$ is the marginal law of $Y$, often called the \emph{climatological forecast} in meteorology. The miscalibration (MCB) term measures how different the forecast $\hat{\bbP}$ is from the optimal calibrated forecast $\bbP_{\smash{\raisebox{-1.2ex}{$\hat{},$}}}$ with the same information, which is the conditional law of $Y$ given $\hat{\bbP}$. Discrimination ability (DSC) quantifies how much more discriminative the optimal calibrated forecast $\bbP_{\smash{\raisebox{-1.2ex}{$\hat{}\,$}}}$ for the given information is in comparison to only marginal information about $Y$. Finally, the uncertainty term (UNC) informs about the general difficulty of the forecasting problem. 

\section{Applications}\label{sec:applications}

In this section, we briefly review the many ways in which proper scoring rules can be used to address computational and statistical considerations in estimation and forecast evaluation contexts.

\subsection{Computational Tractability}

There are often more convenient and effective ways to model a probability distribution than as a member of a standard parametric family. Estimation of such models using the logarithmic score in the form of maximum likelihood can, however, present an often insurmountable computational burden. Proper scoring rules can provide computationally viable alternatives.

\subsubsection{Non-normalized Densities}\label{sec:nonnorm}

Although the logarithmic score yields closed-form expressions for many standard parametric families, it is often intractable for more complicated ones. This occurs frequently in statistics \citep[see for example][]{Fondeville2018} and especially in some modern machine learning applications such as \emph{generative diffusion models} \citep{song2019}, where we need to model probability densities more flexibly, say as $p_{\theta}(\bx) \propto \exp \eta_{\theta}(\bx)$, where $\eta_{\theta}$ could be a neural network. Estimating $\theta$ via maximum likelihood for such complicated models would require evaluating the normalization constant $c_{\theta} = \int \exp \eta_{\theta}(\bx) \D\bx$, which is prohibitively expensive even for moderately high dimensions. The Hyv\"{a}rinnen score allows for a much more tractable estimation of such models by circumventing the need to evaluate the normalization constant $c_{\theta}$ and reducing the problem to minimizing the empirical risk corresponding to
\begin{equation*}
    \Delta_{\bx} \eta_{\theta}(\bx) + \frac{1}{2}\|\nabla_{\bx}\eta_{\theta}(\bx)\|^{2}
\end{equation*}
which depends only on $\eta_{\theta}$. 
In this way, it complements methods like Langevin dynamics, which allow sampling from $p_{\theta}$ using only $\eta_{\theta}$.
This makes it particularly convenient for many statistical applications where normalization can be an issue such as estimation of high-dimensional graphical models \citep{Lin2016}, causal discovery \citep{Rolland2022} and estimation of point processes \citep{Cao2024}. In Bayesian statistics, the Hyv\"{a}rinen score has been used for likelihood-free inference \citep{Pacchiardi2022}, model selection \citep{dawid2015a}, the derivation of objective priors \citep{Leisen2020}, model comparison \citep{Shao2019} and sampling from Markov chain Monte Carlo and variational inference approaches \citep{Zhang2018, Modi2023}. The recent emergence of diffusion models as a powerful approach to generative modeling of images, audio, video and text, is also partly due to advancements in the implementation of score matching \citep[see for example][]{Dhariwal2021}.

The pseudo-spherical scoring rules also enable estimation with non-normalized densities due to their homogeneity and for this reason, they have been applied to estimating distributions on discrete spaces \citep{Takenouchi2015} and 
training energy-based models \citep{Yu2021}.

\subsubsection{Transformation Models}

In many settings in statistics and machine learning, such as that of \emph{structural equation models} \citep{Dimitruk2007} or \emph{implicit generative models} \citep{Mohamed2017}, we model probability distributions $\{\bbP_{\theta}\}_{\theta \in \Theta}$ (or conditional distributions $\{\bbP_{\theta,\bx}\}_{\theta \in \Theta}$) as the laws of parameterized transformations $g_{\theta}(\bN)$ (or conditionally as $g_{\theta}(\bx, \bN)$) for $\theta \in \Theta$, where $\bN$ is some standard random vector such as a multivariate standard Gaussian, instead of a parametric family of densities $\{p_{\theta} : \theta \in \Theta\}$. Fitting such probability distributions by estimating $\theta$ using maximum likelihood via a change of variables would require potentially expensive Jacobian computations and is difficult to implement if the transformations $g_{\theta}$ are not bijective or do not have a closed-form inverse. The original approach to \emph{generative adversarial networks} (GANs) achieved this via an innovative min-max formulation which however complicates training \citep{Goodfellow2014}. Kernel scores have proven to be a viable alternative approach to training such models since they can be estimated directly by sampling from $\bN$ as
\begin{equation*}
    \hat{S}(\bbP_{\theta, \bx}, y) = \frac{1}{m}\sum_{i=1}^{m} h(g_{\theta}(\bx, \bN_{i}), y) - \frac{1}{2m(m-1)}\sum_{i =1}^{m}\sum_{i': i' \neq i} h(g_{\theta}(\bx, \bN_{i}), g_{\theta}(\bx, \bN_{i'}))
\end{equation*}
in the conditional case, for example, where $\{\bN_{i}\}_{i=1}^{m}$ are independent realizations of $\bN$, without requiring the density or an unstable adversarial formulation like GANs \citep{Bouchacourt2016, Gritsenko2020}. 

Consequently, they have been found to be very useful for probabilistic forecasting \citep{Pacchiardi2024}, estimating conditional transformation models \citep{Hothorn2014}, postprocessing of numerical weather prediction models \citep{Chen2024}, Bayesian generalized likelihood-free inference \citep{Pacchiardi2024b}, extrapolating neural networks \citep{ShenMeinshausen2023}, uncertainty quantification of dynamical systems \citep{Buelte2024}, training neural SDEs \citep{Issa2023}, representation learning \citep{Vahidi2024} and even estimating generative diffusion models \citep{Bortoli2025}. Although \citet{GneitingRafteryETAL2005} was the first to apply kernel score minimization for estimating a parametric family, the earliest instance of fitting conditional transformation models appears to be \citet{Hothorn2014} and \citet{Bouchacourt2016} (for neural networks). Non-conditional transformation models based on neural networks appear to have been first estimated using kernel scores by \citet{Dziugaite2015} and \citet{Li2015}. 

\subsection{Statistical considerations}\label{sec:stat-performance}

In the literature on forecast evaluation, many proper scoring rules have been constructed with the aim of detecting specific forecast deficiencies, or for targeting specific types of outcomes such as high impact events. 
Although these scoring rules have been shown to be useful in evaluation, they have not been widely used in estimation contexts. 

\subsubsection{Sensitivity and Robustness}

Many scoring rules have been proposed in order to be sensitive or robust to particular statistical properties of the observations. For example, \citet{Scheuerer2015} proposed the variogram score in order to remedy the relative insensitivity of the energy scores to the correlation structure of multivariate distributions leading to poor discrimination ability \citep{Pinson2013}. In contrast, \citet{Bolin2023} considered a setting where observations are drawn from different distributions, where they show that common scoring rules such as the CRPS overweigh observations coming from more dispersed distributions. Motivated by local geometry considerations (see Section \ref{sec:geometry}), they proposed the \emph{scaled CRPS} to correct for this relative oversensitivity to observation variance which can be understood as a nuisance parameter. Another modification known as the \emph{scaled weighted CRPS} has been proposed by \citet{Olafsdottir2024} in this setting to account for extremes.

\subsubsection{Weighted Scoring Rules}

In many statistical applications such as risk management \citep{holzman2017} and high impact weather prediction \citep{allen2023a}, certain outcomes are of considerably greater consequence to the user than others and naturally, we want the scoring rule to weigh forecast performance for such outcomes accordingly. To achieve this, \citet{allen2023a} propose using kernel scores with kernels modified using \emph{threshold-weighting} or 
\emph{vertical-rescaling}, that is, $h'(x, y) = h(v(x), v(y))$ and $h'(x, y) = h(x, y)w(x)w(y)$, respectively, by using a chaining function $v: \cY \to \cY$ or weighting function $w:\cY \to \bbR$. By choosing $h$ to be a translation invariant kernel, we can control the weight exclusively through the chaining and weighting functions. This generalizes the threshold-weighted CRPS introduced by \citet{GneitingRanjan2011a}, which also introduced a quantile-weighted version of the CRPS. Towards the same end, \citet{diks2011} propose weighted versions of the logarithmic score based on which \citet{holzman2017} propose a different approach to modifying scoring rules, which generally does not lead to kernel scores even if the original score was a kernel score. 

For binary outcomes, \citet{Buja2005} argued that scoring rules should be tailored to the relative cost implications of different types of errors. In a decision-theoretic context, \citet{Johnstone2011} proposed \emph{tailored scoring rules} which are proper scoring rules adapted to the utility function of the decision maker.

\section{Limitations and Extensions}\label{sec:extensions}

Proper scoring rules are versatile and important tools for estimation and forecast evaluation. However, they also have limitations. Proper scoring rules cannot assess tail properties of distributions. For example, under weak conditions on the scoring rule $S$, \citet{BrehmerStrokorb2019} demonstrate that for a given distribution $\bbP$ on $\bbR$, one can always find a distribution $\bbQ$ such that the divergence $d(\bbQ,\bbP) = S(\bbQ,\bbP) - S(\bbP,\bbP)$ is arbitrarily small but $\bbP$ and $\bbQ$ are not tail equivalent. This negative result has further been strengthened by \citet{AllenKohETAL2024}.  

In survival analysis, the classical approach to inference based on maximum likelihood (i.e. the logarithmic score) is complicated by the fact that observations are censored because this often renders the likelihood intractable. Consequently, many approaches based on alternative scoring rules such as the survival-CRPS \citep{Avati2020} and integrated Brier score \citep{graf1999} have been devised over the years, see also \citet{dawid2014}. Unfortunately, these have been found to be improper by \citet{rindt2022}, implying that they are unfit for evaluation or estimation. \citet{rindt2022} establish the propriety of the classical right-censored log-likelihood and propose a tractable approach to survival analysis based on it. Moreover, \citet{yanagisawa2023} proposed proper extensions of the Brier and ranked probability scores to survival analysis.

\citet{Hanson2003} has proposed an interesting extension of proper scoring rules called \emph{market scoring rules} for improving prediction markets. The basic premise of a prediction market is as follows: for a binary event of interest, say if it will rain on a given day, one can float a security in the prediction market which pays one dollar if (and when) the event takes place and nothing if it does not. 
The price of such a security can then be interpreted as the probability of the event (as forecasted by the market). However, obtaining the market forecast for a continuous quantity, such as the amount of rainfall on a given day, or a joint forecast for a multitude of related events, would require a very large number of such binary securities. Consequently, each security will barely have any buyers or sellers and the market will not function. This is known as the \emph{thin market problem}. Hanson's solution is to use instead a single security based on a proper scoring rule $S$ on the possible probability distributions of the outcomes and an \emph{automated market maker} (AMM) which functions as follows: if a participant disagrees with the AMM forecast $\bbP_{t}$ at time $t$, he can replace it with his own $\bbP_{t+1}$ and be paid\footnote{We use scores in their positive orientation here.} $S(\bbP_{t+1}, y) - S(\bbP_{t}, y)$ if the random quantity realized is $y$. The idea has received renewed interest due to recent developments concerning cryptocurrencies. Interesting further developments include \citet{Hanson2007}, \citet{Abramowicz2007} and \citet{Berg2009}.

\section*{Acknowledgements}

The authors would like to thank an anonymous reviewer, Sam Allen, Philip Dawid, Ritabrata Dutta, Sebastian Engelke, Rafael Frongillo, Tilmann Gneiting, Daniel Hug, Pierre Jacob, Alexander Jordan, Kristof Kraus, Andrea Nava, Romain Pic, Mathias Trabs, and Xinwei Shen for comments, inspiring and fruitful discussions, and useful pointers to the literature.

\bibliography{references}

\appendix

\section*{Appendix}

\section{Proofs for Section \ref{sec:properscoringrules}}

\begin{proof}[Proof of Theorem \ref{thm:regression}]
    If $S$ is a (strictly) proper scoring rule, then, by Fubini's theorem,
    \begin{align*}
        \bbE[S(\bbP^{X}, Y)] 
        &= \int_{\cX} \int_{\cY} S(\bbP^{x}, y) \D\bbP_{Y|X = x}(y) \D\bbP_{X}(x) \\
        &= \int_{\cX} S(\bbP^{x}, \bbP_{Y|X = x}) \D\bbP_{X}(x)
    \end{align*}
    is minimized if (and only if) $\bbP^{x} = \bbP_{Y| X= x}$ for $x \in \cX$ $\bbP_{X}$-almost surely. For the converse, we only need to take $X$ constant.
\end{proof}

\setcounter{definition}{0}
\begin{remark}[Derivation of Hyv\"{a}rinen score entropy]
Clearly,
\begin{align*}
    H(\bbP) 
    &=  \int \left[\Delta_{\by}\log p(\by) + \frac{1}{2}\|\nabla_{\by} \log p(\by)\|^{2}\right]p(\by) \D\by, \\
    &=\int [\nabla_{\by} \cdot \nabla_{\by}]\log p(\by)p(\by) \D\by + \frac{1}{2}\int \|\nabla_{\by} \log p(\by)\|^{2}p(\by) \D\by.
\end{align*}
Using the divergence theorem,
\begin{align*}
    \int [\nabla_{\by} \cdot \nabla_{\by}]\log p(\by)p(\by) \D\by 
    &=\int \nabla_{\by}\cdot[\nabla_{\by}\log p(\by)p(\by)] \D\by - \int [\nabla_{\by}\log p(\by)]\cdot [\nabla_{\by}p(\by)] \D\by \\
    &=\oint \nabla_{\by}\log p(\by)p(\by) \D S(\by) - \int [\nabla_{\by}\log p(\by)]\cdot [\nabla_{\by}\log p(\by)] p(\by)\D\by \\
    &= 0 - \int \|\nabla_{\by}\log p(\by)\|^{2} p(\by)\D\by
\end{align*}
where the surface integral is zero because $\|\nabla_{\by}\log p(\by)\| \to 0$ as $\by \to 0$. The conclusion follows.
\end{remark}

The entropy $H$ is said to be \emph{twice differentiable} at $\bbP \in \sP$ if it is differentiable at $\bbP \in \sP$ with gradient $\nabla_{\bbP}H$ and there exists a $\sP \otimes \sP$-integrable function $\nabla^{2}_{\bbP}H: \cY \times \cY \to \overline{\bbR}$ such that
\begin{multline}\label{eq:hessian}
    \lim_{\alpha \to 0^+} \frac{1}{\alpha^{2}}\Big[H((1-\alpha)\bbP + \alpha \bbQ) - H(\bbP) - \alpha \langle \nabla_{\bbP}H, \bbQ - \bbP \rangle\Big]\\ = \frac{1}{2} \langle \nabla^{2}_{\bbP}H (\bbQ - \bbP), \bbQ-\bbP \rangle\\ = \frac{1}{2}\iint \nabla^2_\bbP H (x,x')\D (\bbQ - \bbP)(x)\D(\bbQ - \bbP)(x').
\end{multline}
Since the entropy $H$ is concave, the second derivative $\nabla^2_\bbP H$ is negative definite in the sense that the right-hand side of \eqref{eq:hessian} is always less or equal to zero. This can be seen by using that
\[
H((1-\alpha)\bbP + \alpha \bbQ) \le H(\bbP) + \langle \nabla_\bbP H, (1-\alpha)\bbP + \alpha \bbQ - \bbP\rangle
\]
due to the supergradient inequality \eqref{eq:supergradient}.

\begin{proof}[Proof of Theorem \ref{thm:div-hessian}]
    If the entropy $H: \sP \to \overline{\bbR}$ is twice-differentiable, then the function $h(t) = H(t \bbQ + (1-t) \bbP)$ can be expressed as a 1st order Taylor series expansion around $t = 0$ as $h(t) = h(0) + t h'(0) + r_{2}(t)$
where the remainder $r_{2}(t)$ can be expressed in the integral form as
\begin{equation}\label{eqn:remainder}
    r_{2}(t) = \frac{1}{2} \left\langle \bbQ - \bbP, \left[ \int_{0}^{1} (t-u)\nabla_{u\bbQ + (1-u)\bbP}^{2}H \D u \right] (\bbQ - \bbP)\right\rangle.
\end{equation}
Evaluating the Taylor series and the remainder term (\ref{eqn:remainder}) for $t = 1$ leads to the desired result.
\end{proof}

\section{Proofs for Section \ref{sec:kernel}}

\begin{proof}[Proof of Theorem \ref{thm:kernel-score}]
\citet[Proposition 20, Remark 21]{SejdinovicSriperumbudurETAL2013} imply that for all $\bbP,\bbQ \in \sP_h$, $\iint h(x, y) \D\bbP(x)\D\bbQ(y)$ is finite. 
 The entropy $H(\bbP)$ is concave, since
    \begin{multline}\label{eqn:strict-concave} 
        0 \leq H(\lambda \bbP + (1-\lambda) \bbQ) - \lambda H(\bbP) - (1-\lambda) H(\bbQ) 
        = -\lambda(1-\lambda)\iint h(x, x') \D(\bbP - \bbQ)(x) \D(\bbP - \bbQ)(x')
    \end{multline}
    for every $\bbP, \bbQ \in \sP_h$ and $\lambda \in [0, 1]$, if $h$ is a conditionally negative definite kernel by taking $\bbP, \bbQ$ to be discrete measures and using approximation arguments. Furthermore, strict concavity of $H$ (that is, equality in (\ref{eqn:strict-concave}) if and only if $\lambda \in \{0, 1\}$) coincides with $h$ being strongly conditionally negative definite. The conclusion follows by applying Theorem \ref{thm:gneiting-char} to $H$, where the subgradient can be computed using \eqref{eq:diff_entropy}.
\end{proof}

\begin{proof}[Proof of Theorem \ref{thm:kernel-scores-char}]
Suppose that the kernel score is proper. Then $h$ must be conditionally negative definite and \citet[Proposition 3.3.2]{berg1984} implies that 
there exists a Hilbert space $\sH$, a subset $\{\psi_{y}\}_{y \in \cY} \subset \sH$ and a function $f: \cY \to \bbR$, such that
\begin{equation}\label{eq:hpsi}
 h(x, y) = \|\psi_{x} - \psi_{y}\|_{\sH}^{2} + f(x) + f(y)
\end{equation}
for every $x, y \in \cY$. This implies the first equality in \eqref{eqn:kernel-score-embedding}. The second equality is a straight forward computation, and the statement on strict propriety follows from the last expression of the divergence at \eqref{eqn:kernel-score-embedding}. The converse follows since $h$ at \eqref{eq:hpsi} is conditinally negative definite. 
\end{proof}

\begin{remark}\label{rem:wlog-kernel}
   Requiring $h$ to be non-negative does not lead to any loss of generality. Indeed, by \eqref{eq:hpsi}, the entropy of any $h$ can be written as 
    \begin{equation*}
        H(\bbP) = \frac{1}{2}\iint h(x, y) \D\bbP(x)\D\bbP(y) = \frac{1}{2}\iint \|\psi_{x} - \psi_{y}\|_{\sH}^{2} \D\bbP(x)\D\bbP(y) + \int f(x) \D\bbP(x)
    \end{equation*}
    where the linear term $\int f(x) \D\bbP(x)$ is inconsequential and the kernels $h$ and the non-negative kernel $(x, y) \mapsto \|\psi_{x} - \psi_{y}\|_{\sH}^{2}$ lead to the same kernel score.
\end{remark}

For the proof of Theorem \ref{thm:sym-char}, we need the following definition. The \emph{supergradient map} $\bbP \mapsto h_{\bbP}$ is \emph{weakly continuous} if $\lim_{t \to 0+} \langle h_{\bbP + t\Delta\bbP}, \bbQ\rangle = \langle h_{\bbP}, \bbQ \rangle$ for $\bbP, \bbP', \bbQ \in \sP$ where $\Delta\bbP = \bbP' - \bbP$.

\begin{proof}[Proof of Theorem \ref{thm:sym-char}]
By Theorem \ref{thm:gneiting-char}, there exists a concave function $H: \sP \to \bbR$ and supergradients $\{h_{\bbP}\}_{\bbP \in \sP}$ such that 
    \begin{equation*}
        S(\bbP, y) = H(\bbP) + \langle h_{\bbP}, \delta_{y} - \bbP \rangle 
    \end{equation*}
    for every $y \in \cY$. For $\bbP,\bbQ \in \sP$, using $d(\bbP, \bbQ) = d(\bbQ, \bbP)$, we obtain
    \begin{equation}\label{eqn:symmetry-div}
        H(\bbQ) - H(\bbP) = \frac{1}{2}\langle h_{\bbP} + h_{\bbQ},  \bbQ - \bbP \rangle.  
    \end{equation}
   By (\ref{eqn:symmetry-div}), the Gateaux derivative of $H$ at $\bbQ$ in the direction $\Delta\bbQ = \bbQ' - \bbQ$, where $\bbQ' \in \sP$, exists and can be evaluated as
    \begin{equation*}
        \lim_{t \to 0} \frac{H(\bbQ + t\Delta\bbQ) - H(\bbQ)}{t} = \frac{1}{2} \lim_{t \to 0} ~\langle h_{\bbQ} + h_{\bbQ + t\Delta\bbQ}, \Delta\bbQ \rangle = \langle h_{\bbQ}, \Delta\bbQ\rangle
    \end{equation*}
    using the weak continuity of $\bbP \mapsto h_{\bbP}$.
    Evaluating the Gateaux derivative of both sides of (\ref{eqn:symmetry-div}) with respect at $\bbQ$ in the direction $\Delta\bbQ = \bbQ' - \bbQ$ gives
    \begin{align*}
        \langle h_{\bbQ}, \Delta\bbQ\rangle
        &= \lim_{t \to 0} \frac{1}{2t}\Big[  \langle h_{\bbP} + h_{\bbQ + t\Delta\bbQ},  \bbQ + t\Delta\bbQ - \bbP \rangle - \langle h_{\bbP} + h_{\bbQ},  \bbQ - \bbP \rangle \Big] \\
        &= \lim_{t \to 0} \frac{1}{2t}\Big[  \langle h_{\bbP} + h_{\bbQ + t\Delta\bbQ},  \bbQ - \bbP \rangle - \langle h_{\bbP} + h_{\bbQ},  \bbQ - \bbP \rangle \Big] \\
        &+\, \lim_{t \to 0} \frac{t}{2t}\Big[  \langle h_{\bbP} + h_{\bbQ + t\Delta\bbQ}, \Delta\bbQ \rangle  \Big] \\
        &= \lim_{t \to 0} \frac{1}{2t}\Big[  \langle h_{\bbQ + t\Delta\bbQ} - h_{\bbQ},  \bbQ - \bbP \rangle \Big] + \frac{1}{2}\Big[  \langle h_{\bbP} + h_{\bbQ}, \Delta\bbQ \rangle  \Big] \\
    \end{align*}
    It follows that
    \begin{equation}\label{eqn:gat-diff-1}
        \lim_{t \to 0} \frac{1}{t}\Big[  \langle h_{\bbQ + t\Delta\bbQ} - h_{\bbQ},  \bbQ - \bbP \rangle\Big] = \langle h_{\bbQ} - h_{\bbP},  \Delta\bbQ \rangle
    \end{equation}
    for every $\bbP, \bbQ, \bbQ' \in \sP$. By taking the difference of (\ref{eqn:gat-diff-1}) with $\bbP = \bbP' \in \sP$ and with $\bbP \in \sP$, we get
    \begin{equation}\label{eqn:gat-diff}
        \lim_{t \to 0} \frac{1}{t}\Big[  \langle h_{\bbQ + t\Delta\bbQ} - h_{\bbQ},  \Delta\bbP  \rangle\Big] = \langle h_{\bbP'} - h_{\bbP},  \Delta\bbQ \rangle.
    \end{equation}
    for every $\bbP, \bbP', \bbQ, \bbQ' \in \sP$ where $\Delta\bbP = \bbP' - \bbP$. Using (\ref{eqn:gat-diff}) on itself, we get
    \begin{align*}
        \langle h_{\bbP'} - h_{\bbP},  \Delta\bbQ \rangle 
        &= \lim_{t \to 0} \frac{1}{t}  \langle h_{\bbQ + t\Delta\bbQ} - h_{\bbQ},  \Delta\bbP \rangle \tag*{By (\ref{eqn:gat-diff})}\\
        &= \lim_{t \to 0} \frac{1}{t}\left[ \lim_{s \to 0} \frac{1}{s} \Big[\langle h_{\bbP + s\Delta\bbP} - h_{\bbP},  \bbQ + t\Delta\bbQ - \bbQ \rangle\Big]\right] \tag*{By (\ref{eqn:gat-diff})}\\
        &= \lim_{s \to 0} \frac{1}{s} \langle h_{\bbP + s\Delta\bbP} - h_{\bbP},  \Delta\bbQ\rangle \\
        &= \langle h_{\bbQ'} - h_{\bbQ},  \Delta\bbP \rangle.
    \end{align*}
    Thus, 
    \begin{equation}\label{eqn:sym-gat}
        \langle h_{\bbP'} - h_{\bbP},  \Delta\bbQ \rangle = \langle h_{\bbQ'} - h_{\bbQ},  \Delta\bbP \rangle
    \end{equation}
    for every $\bbP, \bbP', \bbQ, \bbQ' \in \sP$.
    
    Let $x_{0}, y_{0} \in \cX$. We expand the dual mappings into integrals twice as follows:
    \begin{align*}
        \langle h_{\bbP'} - h_{\bbP},  \Delta\bbQ \rangle 
        &= \int \Big[ h_{\bbP'}(y) - h_{\bbP}(y) \Big] \D(\Delta\bbQ)(y) \\
        &= \int \Big[ h_{\bbP'}(y) - h_{\bbP}(y) - h_{\bbP'}(y_{0}) + h_{\bbP}(y_{0}) \Big] \D(\Delta\bbQ)(y) \\
        &= \int \langle h_{\bbP'} - h_{\bbP}, \delta_{y} - \delta_{y_{0}} \rangle \D(\Delta\bbQ)(y) \\
        &= \int \langle h_{\delta_{y}} - h_{\delta_{y_{0}}}, \Delta\bbP \rangle  \D(\Delta\bbQ)(y) \tag*{By (\ref{eqn:sym-gat})}\\
        &= \iint \langle h_{\delta_{y}} - h_{\delta_{y_{0}}}, \delta_{x} - \delta_{x_{0}} \rangle \D(\Delta\bbP)(x)\D(\Delta\bbQ)(y)\\
        &= \iint \langle h_{\delta_{y}}, \delta_{x} \rangle \D(\Delta\bbP)(x)\D(\Delta\bbQ)(y).
    \end{align*}
    Thus,
    \begin{equation*}
        \langle h_{\bbP'} - h_{\bbP}, \Delta\bbQ \rangle = \int_{\cX} \int_{\cX} \langle h_{\delta_{y}}, \delta_{x} \rangle \D(\Delta\bbP)(x)\D(\Delta\bbQ)(y)
    \end{equation*}
    and by symmetry, 
    \begin{equation*}
        \langle h_{\bbQ'} - h_{\bbQ}, \Delta\bbP \rangle = \int_{\cX} \int_{\cX} \langle h_{\delta_{y}}, \delta_{x} \rangle \D(\Delta\bbQ)(x)\D(\Delta\bbP)(y).
    \end{equation*}
    By (\ref{eqn:sym-gat}) and the Fubini-Tonelli theorem, we derive the integral representation
    \begin{equation}\label{eqn:integral-repn}
        \langle h_{\bbQ'} - h_{\bbQ}, \Delta\bbP \rangle =  \langle h_{\bbP'} - h_{\bbP},  \Delta\bbQ \rangle = \int_{\cX} \int_{\cX} K(x, y) \D(\Delta\bbP)(x)\D(\Delta\bbQ)(y)
    \end{equation}
    where $K(x, y) = \tfrac{1}{2}\left[\langle h_{\delta_{y}}, \delta_{x} \rangle + \langle h_{\delta_{x}}, \delta_{y} \rangle\right] = \tfrac{1}{2}[h_{\delta_{y}}(x) + h_{\delta_{x}}(y)]$ for $x, y \in \cY$.
    By definition, $K$ is symmetric. Due to the integral representation \eqref{eqn:integral-repn}, we can write
    \begin{equation*}
         \int_{\cX} \left[ h_{\bbP'}(y) - h_{\bbP}(y) - \int_{\cX} K(x, y) \D(\Delta\bbP)(x)\right]\D(\Delta\bbQ)(y) = 0
    \end{equation*}
    for every $\bbQ, \bbQ' \in \sP$. By taking $\bbQ = \delta_{y}$ and $\bbQ' = \delta_{y'}$ for $y, y' \in \cX$ with $y'$ fixed, we can see that $h_{\bbP'}(y) - h_{\bbP}(y) - \int_{\cX} K(x, y) \D(\Delta\bbP)(x) = c$ for every $y \in \cX$ some constant $c \in \mathbb{R}$. By fixing $\bbP' \in \sP$, we have
    \begin{equation*}
          h_{\bbP}(y) - \int_{\cX} K(x, y) \,d\bbP(x) = C(y)
    \end{equation*}
    for some constant function $C: \cX \to \mathbb{R}$. It follows that
    \begin{equation*}
        \langle h_{\bbP}, \bbQ \rangle - \int_{\cX}\int_{\cX} K(x, y) \D\bbP(x)\D\bbQ(y) = l(\bbQ)
    \end{equation*}
    for some linear function $l: \sP \to \overline{\bbR}$. Here, a function $l: \sP \to \overline{\bbR}$ is said to be \emph{linear} if $l(\alpha\bbP + (1-\alpha)\bbQ) = \alpha \cdot l(\bbP) + (1-\alpha) \cdot l(\bbQ)$ for every $\alpha \in [0, 1]$ and $\bbP, \bbQ \in \sP$. 
    
    By substituting in (\ref{eqn:symmetry-div}), we get (using symmetry of $K$)
    \begin{align*}
        2[H(\bbQ) - H(\bbP)] &= \langle h_{\bbP} + h_{\bbQ},  \bbQ - \bbP \rangle \\
        &= \langle h_{\bbP},  \bbQ \rangle - \langle h_{\bbP},  \bbP \rangle + \langle h_{\bbQ},  \bbQ \rangle - \langle h_{\bbQ},  \bbP \rangle \\
        &= \int\int K(x, y) \D\bbP(x)\D\bbQ(y) - \int\int K(x, y) \D\bbP(x)\D\bbP(y) \\
        &\quad + \int\int K(x, y) \D\bbQ(x)\D\bbQ(y) - \int\int K(x, y) \D\bbQ(x)\D\bbP(y)\\
        &\quad + l(\bbQ) - l(\bbP) + l(\bbQ) - l(\bbP) \\
        &=  \int\int K(x, y) \D\bbQ(x)\D\bbQ(y) + 2l(\bbQ)  - \int\int K(x, y) \D\bbP(x)\D\bbP(y) - 2l(\bbP).
    \end{align*}
    It follows that
    \begin{equation*}
        H(\bbP) = \frac{1}{2}\int\int K(x, y) \D\bbP(x)\D\bbP(y) + l(\bbP) + C
    \end{equation*}
    for some $C \in \bbR$. The addition of a linear function to an entropy leads to the same score, so we might as well write \begin{equation*}
        H(\bbP) = \frac{1}{2}\int\int K(x, y) \D\bbP(x)\D\bbP(y),
    \end{equation*}
    and the claim is shown.
\end{proof}

\begin{remark}[Why the quadratic score on $\bbR^{d}$ is not a kernel score.]
Assume to the contrary. Then there exists a kernel $k: \bbR^{d} \times \bbR^{d} \to \bbR$ such that $$\int (p(\by) - q(\by))^{2} \D \by = \int k(\by, \by')(p(\by) - q(\by))(p(\by') - q(\by')) \D \by \D \by'$$ for all $p, q$. Since these are quadratic forms in $p - q$, we can use the parallelogram law to derive
$$\int (p_{1}(\by) - q_{1}(\by))(p_{2}(\by) - q_{2}(\by)) \D \by = \int k(\by, \by')(p_{1}(\by) - q_{1}(\by))(p_{2}(\by') - q_{2}(\by')) \D \by \D \by'$$
for any densities $p_{1}, p_{2}, q_{1}, q_{2}$ on $\bbR^{d}$.
Let $A$ be a compact subset of $\bbR^{d}$ and $\{e_{j}\}_{j=1}^{\infty}$ be an orthonormal basis of the space of functions $f: A \to \bbR$ satisfying $\int_{A} f(\by) \D\by = 0$. By taking $$p_{1}(\by) = \frac{\max(e_{i}(\by), 0)}{2\int|e_{i}(\by)|\D\by}, \quad q_{1}(\by) = \frac{\max(-e_{i}(\by), 0)}{2\int|e_{i}(\by)|\D\by}, \quad p_{2}(\by) = \frac{\max(e_{j}(\by), 0)}{2\int|e_{j}(\by)|\D\by}, \quad q_{2}(\by) = \frac{\max(-e_{j}(\by), 0)}{2\int|e_{j}(\by)|\D\by},$$
we derive
$$\delta_{ij} = \int_{A} e_{i}(\by)e_{j}(\by) \D\by = \int_{A \times A} k(\by, \by') e_{i}(\by)e_{j}(\by') \D\by \D\by'$$
for every $1 \leq i, j < \infty$. 
Since $\{e_{i}(\by)e_{j}(\by')\}_{i,j=1}^{\infty}$ form a basis for square-integrable functions on $A \times A$, we have
$$\int_{A \times A} |k(\by, \by')|^{2} \D\by\D\by' = \lim_{N \to \infty} \sum_{i,j =1}^{N} \left[ \int_{A \times A} k(\by, \by') e_{i}(\by)e_{j}(\by') \D\by \D\by' \right]^{2} = \lim_{N \to \infty} N = \infty.$$ By Theorem 16, $k$ can be chosen to be a negative-definite kernel of the form $k(\by, \by') = - \langle \psi_{\by}, \psi_{\by'} \rangle$ which means $|k(\by, \by')|^{2}\leq |k(\by, \by)||k(\by', \by')|$. It follows that the entropy of the uniform density $\mathbf{1}_{A}(\by)/\int_{A}\D\by$ is given by $$ \int_{A}\|\psi_{\by}\|^{2} \frac{\D\by}{\int_{A}\D\by} =  \int_{A} |k(\by, \by)| \frac{\D\by}{\int_{A}\D\by} = \infty.$$ This contradicts the fact that the entropy of the quadratic score is finite for the uniform density. The conclusion follows.
\end{remark}

\begin{proof}[Proof of Theorem \ref{thm:spectral-rep-kernel}]
    According to the Levy-Khinchin Theorem \citep[Theorem 3.12.12]{sasvari2013} (in the special case for Hermitian real-valued $f$), a continuous function $h$ on $\bbR^{d}$ is conditionally negative definite if and only if there exist (a) a scalar $c \in \bbR$, (b) a vector $\mathbf{a} \in \bbR^{d}$, (c) a positive semidefinite matrix $\bB \in \bbR^{d \times d}$, and (d) a $\sigma$-finite Borel measure $\mu$ satisfying $\mu(\{\bzero\}) = 0$, $\mu(A) = \mu(-A)$ for $A \subset \bbR^{d}$ and $\int_{\bbR^{d}} \min(1, \|\mathbf{x}\|^{2}) \D\mu(\mathbf{x}) < \infty$ such that
\begin{equation*}
    h(\mathbf{x}) = c + \langle \mathbf{x}, \bB \mathbf{x} \rangle - \int_{\bbR^{d}} \left[ e^{i\langle \mathbf{x}, \mathbf{u}\rangle} - 1 - \langle \mathbf{x}, \mathbf{u}\rangle \one{\{\|\mathbf{u}\| \leq 1\}} \right] \D\mu(\mathbf{u}).
\end{equation*}
Substituting in $S(\bbP, \bx) = d(\bbP, \delta_{\bx}) + \tfrac{1}{2}h(\bzero)$ and $d(\bbP, \bbQ)$ where $d(\bbP, \bbQ) = -\frac{1}{2} \iint h(\bx - \by) \D(\bbP - \bbQ)(\bx) \D(\bbP - \bbQ)(\by)$ gives the desired result. The special case for negative definite kernel follows from using Bochner's Theorem \citep[Theorem 1.7.4.]{sasvari2013} instead of Levy-Khinchine.
\end{proof}

\begin{proof}[Proof of Theorem \ref{thm:kernel-scores-homogeneous-iso}]
    The measure $\mu$ is supported on $\bbR^{d} \setminus \{\bzero\}$, which we can identify with $(0, \infty) \times \mathbb{S}^{d-1}$ via the transformations $r = \|\bu\| \in (0, \infty)$ and $\sigma = \bu/\|\bu\| \in \mathbb{S}^{d-1}$. By Rohlin's measure disintegration theorem \cite[see][Theorem 6.4]{Simmons2012}, we can write
    \begin{equation*}
        \D\mu(\bu) = \D\nu_{r}(\sigma)\D\rho(r)
    \end{equation*}
    where $\{\nu_{r}\}$ are probability measures on $\mathbb{S}^{d-1}$ for $r \in (0, \infty)$ and $\rho$ is a $\sigma$-finite measure on $(0, \infty)$. In order for $\int_{\bbR^{d}} \min(1, \|\mathbf{x}\|^{2}) \D\mu(\mathbf{x}) < \infty$ to hold, we must have $\int_{0}^{\infty} \min(1, r^{2})\D\rho(r) < \infty$.
    \begin{enumerate}
        \item Suppose that $S(\bbP_{c}, c\by) = c^{\alpha}S(\bbP, \by)$ for every $c > 0$. By Theorem \ref{thm:spectral-rep-kernel}, we have
    \begin{align}
        0 &= (c^{2} - c^{\alpha})(\by - \mathbf{m}_{\bbP})^{\top}\bB (\by - \mathbf{m}_{\bbP}), \label{eqn:lkhin-1}\\ 
        0 &= c^{\alpha}\int_{\bbR^{d}} |e^{i\bu^{\top}\by} - f_{\bbP}(\bu)|^{2} \D\mu(\bu) -\int_{\bbR^{d}} |e^{ic\bu^{\top}\by} - f_{\bbP}(c\bu)|^{2} \D\mu(\bu), \label{eqn:lkhin-2}
    \end{align}
    since the L\`{e}vy-Khinchine decomposition is unique. By doing a change of variables in the second equation, we get
    \begin{equation*}
        \int_{\bbR^{d}} |e^{i\bu^{\top}\by} - f_{\bbP}(\bu)|^{2} \Big[ c^{\alpha}\D\mu(\bu) - \D\mu(\bu/c)\Big]= 0.
    \end{equation*}
   For $\bbP = \delta_{\mathbf{0}}$, this gives (using the fact that $\mu$ is even)
    \begin{multline*}
        \int_{\bbR^{d}} \Big[2 -  e^{i\bu^{\top}\by} - e^{-i\bu^{\top}\by}\Big] \Big[ c^{\alpha}\D\mu(\bu) - \D\mu(\bu/c)\Big]
        \\= 2\int_{\bbR^{d}} \Big[1 -  e^{i\bu^{\top}\by} \Big] \Big[ c^{\alpha}\D\mu(\bu) - \D\mu(\bu/c)\Big]
        = 0.
    \end{multline*}
    Thus,
    \begin{equation*}
        \int_{\bbR^{d}} \Big[1 -  e^{i\bu^{\top}\by} \Big]  c^{\alpha}\D\mu(\bu)
        = \int_{\bbR^{d}} \Big[1 -  e^{i\bu^{\top}\by} \Big] \D\mu(\bu/c).
    \end{equation*}
    Again, using the evenness of $\mu$ we can write this as
    \begin{equation*}
        \int_{\bbR^{d}} \Big[e^{i\bu^{\top}\by} - 1 - \langle \mathbf{y}, \mathbf{u}\rangle \one{\{\|\mathbf{u}\| \leq 1\}} \Big]  c^{\alpha}\D\mu(\bu)
        = \int_{\bbR^{d}} \Big[e^{i\bu^{\top}\by} - 1 - \langle \mathbf{y}, \mathbf{u}\rangle \one{\{\|\mathbf{u}\| \leq 1\}} \Big] \D\mu(\bu/c),
    \end{equation*}
    Since, this is a L\`{e}vy-Khinchine decomposition, we have by uniqueness that the two L\`{e}vy measures must be equal, which gives 
    \begin{equation}\label{eqn:mu-homogeneous}
        \mu(cA) = \mu(A)/c^{\alpha}
    \end{equation}
    for every Borel set $A \subset \bbR^{d}$. 
    From (\ref{eqn:mu-homogeneous}) we get
    \begin{equation*}
        \D\nu_{cr}(\sigma) \D\rho_{c}(r) = \frac{1}{c^{\alpha}} \D\nu_{r}(\sigma) \D\rho(r)
    \end{equation*}
    where $\rho_{c}(A) = \rho(cA)$ for Borel sets $A \subset (0, \infty)$. By integrating over $\sigma \in \mathbb{S}^{d-1}$, we get
    \begin{equation}\label{eqn:integrate-out-sigma}
        \D\rho_{c}(r) = \frac{1}{c^{\alpha}} \D\rho(r)
    \end{equation}
    which implies $\rho((0, cr)) = \frac{1}{c^{\alpha}}\rho((0, r))$ for $c, r > 0$. By fixing $r = 1$ and replacing $c$ with $r$, we get $\rho((0, r)) = \frac{1}{r^{\alpha}}\rho((0, 1))$, which implies that $\rho$ is absolutely continuous with respect to the Lebesgue measure on $(0, \infty)$ and 
    \begin{equation*}
        \D\rho(r) = \smash{\frac{\alpha \rho((0, 1))}{r^{1+\alpha}}} \D r.
    \end{equation*}
    Furthermore, by (\ref{eqn:integrate-out-sigma}), $\D \nu_{cr}(\sigma) = \D\nu_{r}(\sigma)$ which means that $\nu_{r} = \nu$ for $r > 0$, where $\nu = \nu_{1}$. It follows that
    \begin{equation*}
        \D\mu(\bu) = \D\nu(\sigma) \, \frac{\alpha \rho((0, 1))}{r^{1+\alpha}}\, \D r
    \end{equation*}
    where we can write $\rho((0, 1)) = \mu(\mathbb{S}^{d-1}\times (0,1))$. In order for a nonzero measure $\rho$ to satisfy $\int_{0}^{\infty} \min(1, r^{2}) \D\rho(r) < \infty$, we must have $\alpha \in (0, 2)$. \\
    
    From (\ref{eqn:lkhin-1}) and (\ref{eqn:lkhin-2}), we thus have two possibilities: (a) $\alpha = 2$ implying $\alpha \not\in (0, 2)$ and thus, $\mu$ is zero everywhere and (b) $\alpha \neq 2$ implying $\bB = \bzero$ and $\mu$ has the above form with $\alpha \in (0, 2)$ (without ruling out the possibility that it is zero everywhere). 
    
    \item If $S(\bbP_{\bU}, \bU\by) = S(\bbP, \by)$, then 
    \begin{align*}
        0 &= (\by - \mathbf{m}_{\bbP})^{\top}[\bB - \bU^{\top}\bB \bU](\by - \mathbf{m}_{\bbP}), \\
        0 &= \int_{\bbR^{d}} |e^{i\bu^{\top}\by} - f_{\bbP}(\bu)|^{2} \D\mu(\bu) -\int_{\bbR^{d}} |e^{i(\bU^{\top}\bu)^{\top}\by} - f_{\bbP}(\bU^{\top}\bu)|^{2} \D\mu(\bu).
    \end{align*}
    for every rotation $\bU \in \mathrm{SO}(d)$. Arguing as previously, it follows that $\bB = c\mathbf{I}$ for some $c \geq 0$. Moreover, we get that 
    \begin{equation*}
        \D\nu_{r}(\bU\sigma)\D\rho(r) = \D\nu_{r}(\sigma)\D\rho(r).
    \end{equation*}
    Integrating over $r$ reveals that     $\nu_{r}$ is invariant under rotation, that is, $\D\nu_{r}(\bU \sigma) = \D\nu_{r}(\sigma)$, implying $\D\nu_{r}(\sigma) = \Gamma(d/2)\D\sigma/(2\pi^{d/2})$. 
    \end{enumerate}
    The conclusion for energy scores and the CRPS follows from the form of their spectral measures once we combine the above results:
    \begin{equation*}
        \D\mu(\bu) =  \alpha \mu(\mathbb{S}^{d-1}\times (0,1)) \cdot \frac{\Gamma(d/2)}{2\pi^{d/2}} \cdot \frac{\D\sigma \cdot r^{d-1}\D r}{r^{d+\alpha}} \propto \frac{1}{\|\bu\|^{d+\alpha}}\D\bu.
    \end{equation*}
\end{proof}
\end{document}